\documentclass[12pt]{amsart}

\usepackage[english]{babel}
\usepackage[utf8x]{inputenc}
\usepackage[T1]{fontenc} 
\usepackage{amssymb}
\usepackage{longtable}
\usepackage{mathtools}

\usepackage[a4paper,top=3cm,bottom=2cm,left=2cm,right=2cm,marginparwidth=1.75cm]{geometry}

\usepackage{amsmath}
\usepackage{graphicx}
\usepackage[colorinlistoftodos]{todonotes}

\usepackage{amsthm, amsmath}
\newtheorem{theorem}{Theorem}[section]
\newtheorem{corollary}{Corollary}[theorem]
\newtheorem{lemma}[theorem]{Lemma}
\newtheorem{prop}[theorem]{Proposition}
\theoremstyle{definition}

\theoremstyle{remark}
\newtheorem*{remark}{Remark}
\newtheorem{conjecture}[theorem]{Conjecture}

\newtheorem{assumption}[theorem]{Assumption}
\usepackage[colorlinks=true, citecolor=blue, linkcolor=black]{hyperref}

\newcommand{\gal}{\text{Gal}}
\newcommand{\frob}{\text{Frob}}
\DeclareFontFamily{U}{wncy}{}
\DeclareFontShape{U}{wncy}{m}{n}{<->wncyr10}{}
\DeclareSymbolFont{mcy}{U}{wncy}{m}{n}
\DeclareMathSymbol{\Sh}{\mathord}{mcy}{"58} 

\usepackage[square,sort,comma,numbers]{natbib}
\usepackage[labelfont=bf]{caption}

\title{Factorization of measures and applications to the weak Goldfeld conjecture}
\author{Merrick Cai}
\date{August 2021}

\begin{document}

\maketitle
\begin{abstract}
Extending Gross's result, we prove that a certain factorizaton of measures holds for all $p$ and any finite even Dirichlet character $\chi$ of any conductor, rather than only for split $p$ and $\chi$ with conductor a power of $p$. Using this generalization, we find lower bounds on the proportion of imaginary quadratic fields $K$ for which (under certain assumptions on the elliptic curve) a chosen quadratic twist of an elliptic curve $E$ over $K$ has rank $1$. We also find lower and upper bounds for the proportion of quadratic twists with rank $1$ when we vary $D$, the factor we twist by, under the assumption that $\omega$ (the prime factor counting function) is sufficiently close to a Gaussian distribution, as described by Erd\"os-Kac. We apply similar methods to cubic twists, and then derive analogous lower bounds for the proportion of imaginary quadratic fields for which a sextic twist has rank $1$. Lastly, for elliptic curves over $\mathbb{Q}$ satisfying certain assumptions, we find positive lower bounds on the proportion of quadratic twists (over $\mathbb{Q}$) which have rank $0$ and rank $1$, which yields examples of elliptic curves satisfying the weak Goldfeld conjecture.
\end{abstract}
\tableofcontents
\section{Introduction}

\subsection{Algebraic and analytic rank}

Let $E$ be an elliptic curve over $\mathbb{Q}$. The $\mathbb{Q}$-points of $E$ form an abelian group $E(\mathbb{Q}$) called the Mordell-Weil group. Mordell's theorem states that $E(\mathbb{Q})$ is finitely generated, and thus the rank of $E(\mathbb{Q})$ is a well-defined, nonnegative integer. We call the rank of $E(\mathbb{Q})$ the algebraic rank of $E$, and denote it as $r_{alg}(E)$.

However, the algebraic rank is rather difficult to handle. Instead, we may attach the following $L$-function to the elliptic curve $E/\mathbb{Q}$: $$L(E/\mathbb{Q},s)=\prod_{p}L_p(E/\mathbb{Q},s),$$ where $$L_p(E/\mathbb{Q},s)=
\begin{cases}
\left(1-a_p\cdot p^{-s}+p^{1-2s}\right)^{-1}& p\text{ has good reduction},\\
\left(1\pm a_p\cdot p^{-s}\right)^{-1}& p\text{ has multiplicative reduction},\\
1 & p\text{ has additive reduction},
\end{cases}$$ and $a_p=p+1-|E(\mathbb{F}_p)|$ is the trace of the Frobenius element associated to $p$. (In the multiplicative reduction case, the type of reduction determines the sign of the plus/minus.) This $L$-function satisfies a functional equation relating its values at $s$ and $2-s$, and thus its order of vanishing at $s=1$ is of interest. We call the order of vanishing of $L(E/\mathbb{Q},s)$ at $s=1$ the analytic rank of $E$, and denote it as $r_{an}(E)$.

Although the notions of analytic and algebraic rank may seem unrelated, they are not. The famous Birch and Swinnerton-Dyer conjecture \cite{birch1965notes} posits that they are in fact equal.

\begin{conjecture}[Birch and Swinnerton-Dyer]
The algebraic rank is the same as the analytic rank: $r_{alg}(E)=r_{an}(E)$.
\end{conjecture}

The BSD conjecture is still wide open, although significant advances have been made. Some of the strongest known results are due to \cite{taylor1995ring}, \cite{wiles1995modular}, \cite{breuil2001modularity}, \cite{gross1986heegner}, \cite{kolyvagin1989finiteness}, and \cite{kolyvagin2007euler}, and they relate the algebraic and analytic ranks in low rank cases.
\begin{theorem}
If $r_{an}(E)\in\{0,1\}$, then $r_{an}(E)=r_{alg}(E)$.
\end{theorem}
However, it's still unproven as to whether $r_{alg}(E)\in\{0,1\}$ implies that $r_{alg}(E)=r_{an}(E)$.

\subsection{Goldfeld's conjecture}

Elliptic curves can be ordered by a property called height. This property is useful when studying statistics of elliptic curves, since it allows us to formalize the notion of an average: to measure the average of a quantity over all elliptic curves, we can calculate the average over the finitely many elliptic curves with height at most $X$, and then take a limit as $X\to \infty$. The analytic rank of an elliptic curve is one particularly important property that can be studied in this way. Originating from \cite{goldfeld1979conjectures} and \cite{katz1999random}, it is widely believed that among all elliptic curves over $\mathbb{Q}$, the elliptic curves with analytic rank $0$ or $1$ should each have density $50\%$, while elliptic curves with analytic rank greater than $1$ should have density $0$. Recent developments by \cite{bhargava2015ternary}, \cite{bhargava2014majority}, \cite{bhargava2013average}, and others have placed increasingly tighter bounds on the average, putting it closer and closer to the conjectured value of $0.5$; for example, the average rank is bounded below by $0.2068$ and bounded above by $0.885$.

Understanding the average rank over all elliptic curves is rather difficult. We can instead look at one particular family of elliptic curves: the quadratic twists $E_D$ of a fixed elliptic curve $E$. In \cite{goldfeld1979conjectures}, Goldfeld postulated that the average rank of a family of quadratic twists should behave in the same way as the set of elliptic curves over $\mathbb{Q}$.
\begin{conjecture}[Goldfeld]
Let $E$ be an elliptic curve and let $\{E_D\}$ be the family of quadratic twists of $E$ as $D$ varies over the set of fundamental domains. Then $50\%$ of the $E_D$ have analytic rank $0$, $50\%$ have analytic rank $1$, and $0\%$ have analytic rank greater than $1$.
\label{goldfeld}
\end{conjecture}

However, Goldfeld's conjecture is still very open. There is no elliptic curve which has been shown to satisfy Goldfeld's conjecture. We will instead study the following weaker version of Goldfeld's conjecture (see, for example, \cite[Conjecture~1.2]{kriz2019goldfeld}).

\begin{conjecture}[Weak Goldfeld]
As in Conjecture~\ref{goldfeld}, fix $E$ and let $\{E_D\}$ be the family of quadratic twists of $E$. A positive proportion of the $E_D$ have rank $0$ and a positive proportion of the $E_D$ have rank $1$.
\label{weakgoldfeld}
\end{conjecture}

In the last section of this paper, we will prove a result which, given certain conditions on the elliptic curve, guarantee that a positive proportion of its quadratic twists will have rank $0$ and $1$; in addition, we give lower bounds for these proportions.

\subsection{Measures on profinite groups}
\label{subsec:measures}

We follow the exposition in \cite[\S1]{gross1980factorization}. Let $p\in\mathbb{Z}$ be a prime, $\mathbb{Z}_p$ the ring of $p$-adic integers, $\mathbb{Q}_p$ the field of $p$-adic numbers, and $\mathbb{C}_p$ the algebraic closure of $\mathbb{Q}_p$. Now let $\mathbb{D}_p$ be the ring of integral elements in $\mathbb{C}_p$. For a commutative profinite group $G$, we consider its completed group algebra over $\mathbb{D}_p$, $\Lambda_G\coloneqq\mathbb{D}_p[[G]]=\varprojlim_{H\subset G\text{ open}}\mathbb{D}_p[G/H]$. The elements of $\Lambda_G$ are called measures on $G$. We also define $\Lambda_G'$, the total ring of fractions of $\Lambda_G$, as the ring whose elements are $\alpha/\beta$ for $\alpha,\beta\in \Lambda_G$ and $\beta$ is not a zero-divisor.

We define a bilinear pairing between continuous functions $G\to\mathbb{C}_p$ and measures in $\Lambda_G$ by approximating $f$ by locally constant functions and taking a limit, as in \cite{serre1978residu}: $$\langle f,\lambda\rangle=\int_G f\,d\lambda.$$ For $\lambda=\alpha/\beta\in \Lambda_G'$, we extend this pairing by $\langle f,\lambda\rangle\coloneqq \langle f,\alpha\rangle/\langle f,\beta\rangle$. This construction is well-defined since it does not depend on the representation of $\lambda$, and agrees with our previous definition for $\lambda\in\Lambda_G$.

Let $K$ be an imaginary quadratic field. We will primarily consider the case where $G=\gal(K(\mu_{p^\infty})/\mathbb{Q})$ or $G=\gal(K(\mu_{p^\infty})/\mathbb{Q})/\sigma$ where $\sigma$ is complex conjugation, and $f=\chi$ is a (continuous) character from $G$ to $\mathbb{D}_p^\times$.

\subsection{Structure of the paper and main results}

Let $p$ be a prime, $K=\mathbb{Q}(\sqrt{-C})$ an imaginary quadratic field where $p$ splits, and $\chi$ a continuous $p$-adic character of $\gal(K(\mu_{p^{\infty}})/\mathbb{Q})$ which is trivial on complex conjugation. Let $\chi_K$ be the restriction of $\chi$ to $\gal(K(\mu_{p^{\infty}})/K)$, $\epsilon$ the quadratic character modulo $C$, and $\omega$ the Teichm\"uller character. As in \cite{gross1980factorization}, define the measures $\lambda_1,\lambda_2,\lambda_3$ by $\langle\chi,\lambda_1\rangle=L_p(0,\chi_K)$, $\langle\chi,\lambda_2\rangle=L_p(0,\chi\epsilon\omega)$, and $\langle\chi,\lambda_3\rangle=L_p(1,\chi^{-1})$. Motivated by the classical factorization of $L$-series $L(s,(\chi_K)_\infty)=L(s,\chi_\infty\epsilon)L(s,\chi_\infty)$, Gross \cite[Theorem~3.1]{gross1980factorization} derives the factorization of measures $$\lambda_1=\lambda_2\cdot \lambda_3,$$ when $p$ is split in $K$ and $\chi$ is a finite even Dirichlet character whose conductor is a power of $p$. In \S\ref{sec:factorization}, we extend this result in Theorem~\ref{t2.13} to all $p$ (not just split $p$) and any finite even Dirichlet character $\chi$ (with any conductor).

We then turn our attention to elliptic curves. In \S\ref{sec:conventions}, we introduce assumptions on the elliptic curve which will hold for the remainder of the paper. We will assume that $E$ is residually reducible modulo $3$ (Assumption \ref{assump:residuallyreducible}), and we will work with integers $D$ and imaginary quadratic fields $K$ satisfying various divisibility and congruence conditions relating $D$, the conductor of $E$, and the discriminant of $K$ (Assumptions \ref{assump:heegner} and \ref{assump:*}).

In \S\ref{sec:congruences}, we discuss general congruences of $L$-series and Eisenstein series, especially those associated with quadratic characters. We obtain some auxiliary results concerning the congruence of certain modular forms with Eisenstein series, and calculate the Euler factor at $p$ after $p$-depleting (see \S\ref{subsec:stabilize}).

In \S\ref{subsec:varyK}, we use the factorization in Theorem~\ref{t2.13} to arrive at the two key technical results, Theorem~\ref{t6.6} and Theorem~\ref{t6.7}. Under Assumption \ref{assump:*} and the assumption that $D$ satisfies the nonvanishing of a certain class number modulo $3$ (for $D>0$, we need $3\nmid h_{\mathbb{Q}(\sqrt{-3D})}$ and for $D<0$, we need $3\nmid h_{\mathbb{Q}(\sqrt{D})}$), we find a lower bound on the proportion of imaginary quadratic fields $K$ for which $E_D/K$ has rank $1$. In \S\ref{subsec:varyD}, we vary $D$ instead. Assuming that $\omega(n)$ is sufficiently close to a Gaussian distribution, we find bounds on the proportion of $D$ such that $E_D/K$ has rank $1$; these are given in Theorem~\ref{t6.11}.

In \S\ref{sec:highertwists}, we address cubic and sextic twists. In \S\ref{subsec:cubic}, we obtain results similar to \S\ref{sec:congruences} but for cubic twists. Since sextic twists are a composition of a cubic twist and a quadratic twist, we apply our results from \S\ref{subsec:varyK} to obtain similar results on sextic twists in \S\ref{subsec:sextic}. The results, paralleling Theorem~\ref{t6.6} and Theorem~\ref{t6.7}, are given by Theorem~\ref{t8.3} and Theorem~\ref{t8.4}.

Finally, in \S\ref{sec:ranksoverQ}, we positive lower bounds on the proportion of $D$ for which $E_D/\mathbb{Q}$ has rank $0$ and $1$, under similar assumptions. Given the assumptions before, plus the additional assumption that $3\nmid N\coloneqq \text{cond}(E)$, in Theorem~\ref{t9.1} we find that $E_D/\mathbb{Q}$ has rank $0$ for at least $\frac{\phi(N)}{4N}$ of all such $D$, and rank $1$ for at least $\frac{\phi(N)}{4N}$ of all such $D$. As an easy corollary, we conclude Conjecture~\ref{weakgoldfeld} for certain elliptic curves.

\subsection*{Acknowledgements}
The author would like to thank Daniel Kriz for supervising this project, mentoring the author, and providing much needed guidance. The author also thanks Jonathan Love and Professor Andrew Sutherland for many helpful discussions and feedback.

\section{Factorization of measures}
\label{sec:factorization}
We follow the notation in \cite{gross1980factorization}. Let $p$ be a prime, $K=\mathbb{Q}(\sqrt{-C})$ an imaginary quadratic field where $p$ splits, $\chi$ a finite even Dirichlet character on $\gal(K(\mu_{p^{\infty}})/\mathbb{Q})$, $\chi_K$ the restriction of $\chi$ to $\gal(K(\mu_{p^{\infty}})/K)$, and $\chi_\infty$ the composition of $\chi$ with some fixed injection $\overline{\mathbb{Q}}\hookrightarrow \mathbb{C}$. Let $\epsilon$ be the quadratic character modulo $C$ and $\omega$ the Teichm\"uller character. We define $\lambda_1,\lambda_2,\lambda_3$ as in \cite[p.~92]{gross1980factorization}, and obtain the formulas $\langle \chi, \lambda_1\rangle =L_p(0,\chi_K)$, $\langle \chi,\lambda_2\rangle = L_p(0,\chi\epsilon\omega)$, and $\langle \chi,\lambda_3\rangle=L_p(1,\chi^{-1})$, as in \cite[p.~93]{gross1980factorization}.

\subsection{Dirichlet characters with conductors a prime power}
We start with the classical factorization $L(s,(\chi_K)_{\infty})=L(s,\chi_{\infty}\epsilon)L(s,\chi_{\infty})$
and the functional equation for $L(s,\chi)$: $$L(s,\chi)=L(1-s,\overline{\chi})\frac{\Gamma\left(\frac{1-s+a}{2}\right)}{\Gamma\left(\frac{s+a}{2}\right)}\left(\frac{k}{\pi}\right)^{\frac{1}{2}-s}\frac{\tau(\chi)}{i^a\sqrt{k}},$$ where $\Gamma$ is the gamma function, $k$ is the conductor of $\chi$, $\tau=\sum_{n=1}^{k}\chi(n)e^{2\pi i n/k}$ is the Gauss sum, and $a=0$ if $\chi(-1)=1$ while $a=1$ if $\chi(-1)=-1$.

\begin{prop}
$L'(0,(\chi_K)_{\infty})=L(0,\chi_{\infty}\epsilon ) L'(0,\chi_{\infty})$.
\label{p2.1}
\end{prop}
\begin{proof}
By differentiating, $$L'(s,(\chi_K)_\infty)=L'(s,\chi_\infty \epsilon)L(s,\chi_\infty)+L(s,\chi_\infty\epsilon)L'(s,\chi_\infty).$$ Setting $s=0$ yields $$L'(0,(\chi_K)_\infty)=L'(0,\chi_\infty \epsilon)L(0,\chi_\infty)+L(0,\chi_\infty\epsilon)L'(0,\chi_\infty).$$ But since $\chi_\infty(-1)=1\implies a=0$, we have $$L(0,\chi_\infty)=L(1,\overline{\chi_\infty})\frac{\Gamma\left(\frac{1}{2}\right)}{\Gamma(0)}\left(\frac{k}{\pi}\right)^{1/2}\frac{\tau(\chi)}{\sqrt{k}}.$$ Notably, $\frac{1}{\Gamma(0)}=0$, which concludes the result.
\end{proof}

The explicit formulas of Dirichlet and Kronecker are as follows:
\begin{itemize}
    \item $L'(0,(\chi_K)_{\infty})=-\frac{1}{6p^r}\sum_A \chi_\infty (a) \log F^+(a)_\infty$ \cite[p.~91]{gross1980factorization},
    \item $L(0,\chi)=-\sum_{a=1}^{f} \frac{a}{f}\chi(a)=-B_{1,\chi}$ for $f$ the conductor of $\chi$ \cite[p.~88]{gross1980factorization},
    \item $L(1,\chi_{\infty})=-g(\chi_\infty)\sum_A \chi_{\infty}^{-1}(a)\log C^+(a)_\infty$ \cite[p.~91]{gross1980factorization} where $g(\chi)=\frac{1}{f}\sum_{a=1}^{f}\chi(a)e^{2\pi ia/f}$ for $f$ the conductor of $\chi$ \cite[p.~88]{gross1980factorization}.
\end{itemize}

\begin{prop}
\label{p2.2}
$L'(0,\chi_\infty)=-\frac{g(\chi_\infty)}{2}L(1,\overline{\chi_\infty})$.
\end{prop}
\begin{proof}
By differentiating, we have
\begin{align*}
L'(s,\chi)=&L'(1-s,\overline{\chi})\frac{\Gamma\left(\frac{1-s+a}{2}\right)}{\Gamma\left(\frac{s+a}{2}\right)}\left(\frac{k}{\pi}\right)^{\frac{1}{2}-s}\frac{g(\chi)}{i^a\sqrt{k}},\\
&-\frac{1}{2}L(1-s,\overline{\chi})\frac{\Gamma'\left(\frac{1-s+a}{2}\right)}{\Gamma\left(\frac{s+a}{2}\right)}\left(\frac{k}{\pi}\right)^{\frac{1}{2}-s}\frac{g(\chi)}{i^a\sqrt{k}},\\
&-\frac{\Gamma'\left(\frac{s+a}{2}\right)}{2}L(1-s,\overline{\chi})\frac{\Gamma\left(\frac{1-s+a}{2}\right)}{\Gamma\left(\frac{s+a}{2}\right)^2}\left(\frac{k}{\pi}\right)^{\frac{1}{2}-s}\frac{g(\chi)}{i^a\sqrt{k}},\\
&-\log(k/\pi)L(1-s,\overline{\chi})\frac{\Gamma\left(\frac{1-s+a}{2}\right)}{\Gamma\left(\frac{s+a}{2}\right)}\left(\frac{k}{\pi}\right)^{-\frac{1}{2}-s}\frac{g(\chi)}{i^a\sqrt{k}}.\\
\end{align*}
Since $\chi_\infty(-1)=1$, we have $a=0$. Setting $s=0$ yields
\begin{align*}
L'(0,\chi_\infty)=&L'(1,\overline{\chi_\infty})\frac{\Gamma\left(\frac{1}{2}\right)}{\Gamma(0)}\left(\frac{k}{\pi}\right)^{\frac{1}{2}}\frac{g(\chi_\infty)}{\sqrt{k}},\\
&-\frac{1}{2}L(1,\overline{\chi_\infty})\frac{\Gamma'\left(\frac{1}{2}\right)}{\Gamma(0)}\left(\frac{k}{\pi}\right)^{\frac{1}{2}}\frac{g(\chi_\infty)}{\sqrt{k}},\\
&-\frac{\Gamma'(0)}{2}L(1,\overline{\chi_\infty})\frac{\Gamma\left(\frac{1}{2}\right)}{\Gamma(0)^2}\left(\frac{k}{\pi}\right)^{\frac{1}{2}}\frac{g(\chi_\infty)}{\sqrt{k}},\\
&-\log(k/\pi)L(1,\overline{\chi_\infty})\frac{\Gamma\left(\frac{1}{2}\right)}{\Gamma(0)}\left(\frac{k}{\pi}\right)^{-\frac{1}{2}}\frac{g(\chi_\infty)}{\sqrt{k}}.\\
\end{align*}
Note that the Laurent series of $\Gamma(s)$ is $\Gamma(s)=\frac{1}{s}+a_0+a_1s+\dots$ which implies that $$\frac{1}{\Gamma(0)}=0,\hspace{5mm}\frac{\Gamma'(0)}{\Gamma(0)^2}=\left(\frac{\frac{-1}{s^2}+a_1+\dots}{\frac{1}{s^2}+\frac{2a_0}{s}+\dots}\right)\bigg|_{s=0}=-1.$$ Furthermore, $\Gamma(1/2)=\sqrt{\pi}$. Combining these, we find that three of the terms cancel, which yields $L'(0,\chi_\infty)=-\frac{g(\chi_\infty)}{2}L(1,\overline{\chi_\infty}).$
\end{proof}

Using the identity $L(1,\chi_\infty)=-g(\chi_\infty)\sum_{A}\overline{\chi_\infty}(a)\log C^+(a)_\infty$, and the fact that $\tau(\overline{\chi_\infty})g(\chi_\infty)=\frac{\overline{\tau(\chi_\infty)}\tau(\chi_\infty)}{f}=\frac{|\sqrt{f}|^2}{f}=1$, we find that $$L'(0,\chi_\infty)=\frac{1}{2}\sum_A \chi_{\infty}(a)\log C^+(a)_\infty.$$

Now, combining these with the fact that $L(0,\chi_\infty \epsilon)=-B_{1,\chi_\infty \epsilon}$, we find that $$-\frac{1}{6p^r}\sum_A \chi_\infty(a)\log F^+(a)_\infty=(-B_{1,\chi_\infty \epsilon})\left(-\frac{1}{2}\sum_A \chi_\infty (a)\log C^+(a)_\infty\right),$$ or equivalently
\begin{prop}
\label{p2.3}
The equation $$-3p^r B_{1,\chi_{\infty}\epsilon}\sum_{A}\chi_{\infty}(a)\log C^+(a)_\infty = \sum_A \chi_{\infty}(a)\log F^+(a)_\infty$$ holds for all $p$.
\end{prop}
\begin{remark}
This is \cite[(3.5)]{gross1980factorization}, but he only proves it for split $p$.
\end{remark}

Now recall that $C^+(a)_\infty$ and $F^+(a)_\infty$ are $p$-units in the field $M_{p^r}=\mathbb{Q}\left(\cos \frac{2\pi}{p^r}\right)$. Let $E(M_{p^r})$ denote the group of all $p$-units. It is a finitely generated subgroup of $\mathbb{R}^\times$. Now consider the complex vector space $\mathbb{C}\otimes_\mathbb{Z}E(M_{p^r})$. This is isomorphic to the regular representation of $A=\text{Gal}(M_{p^r}/\mathbb{Q})$. Now note that for all $\sigma \in A$, due to transport of structure, we have that 
\begin{align*}
\sigma\left(\sum_A \chi_\infty (a)\otimes_{\mathbb{Z}}C^+(a)_\infty\right)&=\sum_{A}\chi_{\infty}(a)\otimes_{\mathbb{Z}}C^+(\sigma a)_{\infty}=\chi_{\infty}^{-1}(\sigma)\sum_A \chi_{\infty}(\sigma a)\otimes_{\mathbb{Z}}C^+(\sigma a)_{\infty},\\
&=\chi_{\infty}^{-1}(\sigma)\left(\sum_A \chi_\infty (a)\otimes_{\mathbb{Z}}C^+(a)_\infty\right),\\
\sigma\left(\sum_A \chi_\infty (a)\otimes_{\mathbb{Z}}F^+(a)_\infty\right)&=\sum_{A}\chi_{\infty}(a)\otimes_{\mathbb{Z}}F^+(\sigma a)_{\infty}=\chi_{\infty}^{-1}(\sigma)\sum_A \chi_{\infty}(\sigma a)\otimes_{\mathbb{Z}}F^+(\sigma a)_{\infty},\\
&=\chi_{\infty}^{-1}(\sigma)\left(\sum_A \chi_\infty (a)\otimes_{\mathbb{Z}}F^+(a)_\infty\right).\\
\end{align*}
This implies that both $\sum_A \chi_\infty (a)\otimes_{\mathbb{Z}}C^+(a)_\infty$ and $\sum_A \chi_\infty (a)\otimes_{\mathbb{Z}}F^+(a)_\infty$ lie in the $\chi_{\infty}^{-1}$-eigenspace of $\mathbb{C}\otimes_{\mathbb{Z}}E(M_{p^r})$, which is one-dimensional. Therefore $$\tilde{c}\sum_A \chi_\infty (a)\otimes_{\mathbb{Z}}C^+(a)_\infty=\sum_A \chi_\infty (a)\otimes_{\mathbb{Z}}F^+(a)_\infty$$ for some $\tilde{c}\in\mathbb{C}$. Consider the map $$\gamma :\mathbb{C}\otimes_{\mathbb{Z}} E(M_{p^r})\rightarrow \mathbb{C},$$ defined by $\gamma(c\otimes a)=c\log a$. This map is clearly $\mathbb{C}$-linear, so $$\gamma(\tilde{c}\sum_A \chi_\infty (a)\otimes_{\mathbb{Z}}C^+(a)_\infty)=\tilde{c}\gamma(\sum_A \chi_\infty (a)\otimes_{\mathbb{Z}}C^+(a)_\infty).$$ Applying $\gamma$ to both sides of $\tilde{c}\sum_A \chi_\infty (a)\otimes_{\mathbb{Z}}C^+(a)_\infty=\sum_A \chi_\infty (a)\otimes_{\mathbb{Z}}F^+(a)_\infty$ yields that $$\tilde{c}=-3p^r B_{1,\chi_{\infty}\epsilon}.$$ In particular, note that $E(M_{p^r})\subset \overline{\mathbb{Q}}$. We can actually say that
\begin{prop}
\label{p2.4}
As elements of $\overline{\mathbb{Q}}\otimes_{\mathbb{Z}} E(M_{p^r})$, we have $$-3p^rB_{1,\chi_{\infty}\epsilon}\sum_A \chi_\infty (a)\otimes_{\mathbb{Z}}C^+(a)_\infty=\sum_A \chi_\infty (a)\otimes_{\mathbb{Z}}F^+(a)_\infty.$$
\end{prop}
Now take some $\varphi: \overline{\mathbb{Q}}\xhookrightarrow{}\mathbb{C}_p$. Applying $(1\otimes_\mathbb{Z}\log_p)\circ\varphi$ to both sides, we find the following equality in $\mathbb{C}_p$: $$-3p^rB_{1,\chi\epsilon}\sum_A \chi (a)\log_p C^+(a)_p=\sum_A \chi (a)\log_p F^+(a)_p.$$
Now consider the explicit formulas provided by \cite[p.~93]{gross1980factorization}:
\begin{align*}
L_p(0,\chi_K)&=-\frac{1}{3p^r}g(\chi^{-1})\sum_A \chi(a)\log_p F^+(a)_p,\\
L_p(0,\chi \epsilon \omega) &= -B_{1,\chi\epsilon},\\
L_p(1,\chi^{-1})&=-g(\chi^{-1})\sum_A\chi(a)\log_p C^+(a)_p.\\
\end{align*}
Putting these together yields the $p$-adic identity
\begin{prop}
\label{p2.}
$$L_p(0,\chi_K)=L_p(0,\chi\epsilon\omega)L_p(1,\chi^{-1}).$$
\end{prop}
Now, following \cite[p.~93]{gross1980factorization}, there exist measures $\lambda_2,\lambda_3$ such that for any finite even Dirichlet character $\chi$ of conductor $p^r$, we have 
\begin{align*}
\langle \chi, \lambda_2\rangle &= L_p(0,\chi\epsilon\omega),\\
\langle \chi,\lambda_3\rangle &=L_p(1,\chi^{-1}).\\
\end{align*}
Now define a measure $\lambda_1$ given by $$\langle \chi, \lambda_1\rangle = L_p(0,\chi_K).$$ Then we have the equality $$\langle \chi, \lambda_1\rangle = \langle \chi,\lambda_2\rangle \langle \chi,\lambda_3\rangle$$ for all finite even Dirichlet characters $\chi$ with conductor $p^r$. Thus we have that
\begin{theorem}
\label{t2.6}
$\lambda_1=\lambda_2\cdot\lambda_3$.
\end{theorem}

\subsection{Generalization to conductor not a prime power}
We will now work more generally and extend to the remaining cases. We fix $f$ to be some positive integer with at least two distinct prime divisors. We again start with the classical factorization $L(s,(\chi_K)_{\infty})=L(s,\chi_{\infty}\epsilon)L(s,\chi_{\infty})$
and the functional equation for $L(s,\chi)$: $$L(s,\chi)=L(1-s,\overline{\chi})\frac{\Gamma\left(\frac{1-s+a}{2}\right)}{\Gamma\left(\frac{s+a}{2}\right)}\left(\frac{f}{\pi}\right)^{\frac{1}{2}-s}\frac{\tau(\chi)}{i^a\sqrt{f}},$$ where $\Gamma$ is the gamma function, $f$ is the conductor of $\chi$, $\tau=\sum_{n=1}^{f}\chi(n)e^{2\pi i n/f}$ is the Gauss sum, and $a=0$ if $\chi(-1)=1$ while $a=1$ if $\chi(-1)=-1$. 

\begin{prop}
\label{p2.7}
$L'(0,(\chi_K)_{\infty})=L(0,\chi_{\infty}\epsilon ) L'(0,\chi_{\infty})$.
\end{prop}
\begin{proof}
By differentiating, $$L'(s,(\chi_K)_\infty)=L'(s,\chi_\infty \epsilon)L(s,\chi_\infty)+L(s,\chi_\infty\epsilon)L'(s,\chi_\infty).$$ Setting $s=0$ yields $$L'(0,(\chi_K)_\infty)=L'(0,\chi_\infty \epsilon)L(0,\chi_\infty)+L(0,\chi_\infty\epsilon)L'(0,\chi_\infty).$$ But since $\chi_\infty(-1)=1\implies a=0$, we have $$L(0,\chi_\infty)=L(1,\overline{\chi_\infty})\frac{\Gamma\left(\frac{1}{2}\right)}{\Gamma(0)}\left(\frac{f}{\pi}\right)^{1/2}\frac{\tau(\chi)}{\sqrt{f}}.$$ Since $\frac{1}{\Gamma(0)}=0$, which concludes the result.
\end{proof}

We have the following explicit formulas:
\begin{itemize}
    \item $L(0,\chi)=-\sum_{a=1}^{f}\frac{a}{f}\chi(a)=-B_{1,\chi}$ \cite[p.~88]{gross1980factorization},
    \item $L(1,\chi_\infty)=-\frac{\tau(\chi_\infty)}{f}\sum_A \chi_{\infty}^{-1}(a)\log C^+(a)_\infty$ \cite[p.~88]{gross1980factorization} (note that this is for general conductors $f$),
    \item $L'(0,\chi)=-\frac{1}{6fw(\mathfrak{f})}\sum_{c\in\mathcal{C}}\chi(c)\log|E(c)|$ \cite[p.~281]{stark1977class}, where $\chi$ is a ray class character modulo $\mathfrak{f}$, $w(\mathfrak{f})$ is the number of roots of unity equivalent to $1$ mod $\mathfrak{f}$, and $f$ is the conductor of $\chi$.
\end{itemize} Note that $Cl(f)=(\mathbb{Z}/f\mathbb{Z})^\times$, but $A=\text{Gal}(\mathbb{Q}(\cos\frac{2\pi}{f})/\mathbb{Q})=(\mathbb{Z}/f\mathbb{Z})^\times/\pm 1$.

\begin{prop}
\label{p2.8}
$L'(0,\chi_{\infty})=-\frac{1}{2}\sum_A \chi_\infty(a)\log C^+(a)_\infty$.
\end{prop}
\begin{proof}
First note that the classical factorization (with value $a=0$) yields $$L(s,\chi_\infty)\Gamma(s/2)=L(1-s,\chi_\infty^{-1})\Gamma((1-s)/2)\left(\frac{f}{\pi}\right)^{\frac{1}{2}-s}\frac{\tau(\chi_\infty)}{\sqrt{f}}.$$ In particular, consider the left hand side's power series expansion around $s=0$: although $L(s,\chi_\infty)$ vanishes at $s=0$, $\Gamma(0)$ has a pole of order $1$. But since the residue of $\Gamma(s/2)$ is $2$, we have that $$2L'(0,\chi_\infty)=L(1,\chi_\infty^{-1})\Gamma(1/2)\sqrt{f}\sqrt{1/\pi}\frac{\tau(\chi_\infty)}{\sqrt{f}}=\tau(\chi_{\infty})L(1,\chi_\infty^{-1}).$$ Now using Gross's formula, we find that $$L'(0,\chi_\infty)=-\frac{1}{2}\sum_A\chi_\infty (a)\log C^+(a).$$
\end{proof}

\begin{prop}
\label{p2.9}
$L'(0,(\chi_K)_\infty)=-\frac{1}{6f}\sum_A\chi(a)\log F^+(a)$.
\end{prop}
\begin{proof}
Note that we have a quotient homomorphism between the ray class group $\mathcal{C}$ modulo $\mathfrak{f}$, isomorphic to $(\mathbb{Z}/f\mathbb{Z})^\times$, with $A=\text{Gal}(K/\mathbb{Q})=(\mathbb{Z}/f\mathbb{Z})^\times/\pm 1$. Furthermore, $w(\mathfrak{f})=1$ since $K$ is totally real. Then \cite{stark1977class} gives the result $$L'(0,(\chi_K)_\infty)=-\frac{1}{6f}\sum_\mathcal{C}\chi_\infty(c)\log|E(c)|.$$ But note that $$|E(c)|^2=E(c)E(-c)=F_f(a)\overline{F_f(a)}=F_f(a)F_f(-a)$$$$\implies \log F^+(a)=\log F_f(a)F_f(-a)=\log E(a)\overline{E(a)}=\log |E(a)|+\log |E(-a)|,$$ so we have $$L'(0,(\chi_K)_\infty)=-\frac{1}{6f}\sum_A\chi_\infty(a)\log F^+(a).$$
\end{proof}

Now, combining these with the fact that $L(0,\chi_\infty \epsilon)=-B_{1,\chi_\infty \epsilon}$, we find that $$-\frac{1}{6f}\sum_A \chi_\infty(a)\log F^+(a)_\infty=(-B_{1,\chi_\infty \epsilon})\left(-\frac{1}{2}\sum_A \chi_\infty (a)\log C^+(a)_\infty\right),$$ or equivalently
\begin{prop}
\label{p2.10}
The equation $$-3f B_{1,\chi_{\infty}\epsilon}\sum_{A}\chi_{\infty}(a)\log C^+(a)_\infty = \sum_A \chi_{\infty}(a)\log F^+(a)_\infty$$ holds for all $f$.
\end{prop}
\begin{remark}
This is \cite[(3.5)]{gross1980factorization}, but he only proves it for split $p$ and $f=p^r$. Note that $\chi_\infty$ depends on $f$.
\end{remark}

Now note that $C^+(a)_\infty$ and $F^+(a)_\infty$ are units in the field $M_f=\mathbb{Q}\left(\cos \frac{2\pi}{f}\right)$. (In particular, $$\prod_{a\in(\mathbb{Z}/f\mathbb{Z})^\times/\pm 1}(1-e^{2\pi ia/f})=\Phi_f(1)=1,$$ which holds whenever $f$ has at least two distinct prime divisors.) Let $E(M_f)$ denote the group of all units. It is a finitely generated subgroup of $\mathbb{R}^\times$ of rank $|A|-1$, by Dirichlet's unit theorem. Now consider the complex vector space $\mathbb{C}\otimes_\mathbb{Z}E(M_f)$. This is isomorphic to the quotient of the regular representation of $A=\text{Gal}(M_f/\mathbb{Q})$ by the subspace spanned by $(1,1,1,\dots)$. For all $\sigma \in A$, due to transport of structure, we have that 
\begin{align*}
\sigma\left(\sum_A \chi_\infty (a)\otimes_{\mathbb{Z}}C^+(a)_\infty\right)&=\sum_{A}\chi_{\infty}(a)\otimes_{\mathbb{Z}}C^+(\sigma a)_{\infty}=\chi_{\infty}^{-1}(\sigma)\sum_A \chi_{\infty}(\sigma a)\otimes_{\mathbb{Z}}C^+(\sigma a)_{\infty},\\
&=\chi_{\infty}^{-1}(\sigma)\left(\sum_A \chi_\infty (a)\otimes_{\mathbb{Z}}C^+(a)_\infty\right),\\
\sigma\left(\sum_A \chi_\infty (a)\otimes_{\mathbb{Z}}F^+(a)_\infty\right)&=\sum_{A}\chi_{\infty}(a)\otimes_{\mathbb{Z}}F^+(\sigma a)_{\infty}=\chi_{\infty}^{-1}(\sigma)\sum_A \chi_{\infty}(\sigma a)\otimes_{\mathbb{Z}}F^+(\sigma a)_{\infty},\\
&=\chi_{\infty}^{-1}(\sigma)\left(\sum_A \chi_\infty (a)\otimes_{\mathbb{Z}}F^+(a)_\infty\right).\\
\end{align*}
This implies that both $\sum_A \chi_\infty (a)\otimes_{\mathbb{Z}}C^+(a)_\infty$ and $\sum_A \chi_\infty (a)\otimes_{\mathbb{Z}}F^+(a)_\infty$ lie in the $\chi_{\infty}^{-1}$-eigenspace of $\mathbb{C}\otimes_{\mathbb{Z}}E(M_f)$, which is one-dimensional since $A$ is abelian. Therefore $$\tilde{c}\sum_A \chi_\infty (a)\otimes_{\mathbb{Z}}C^+(a)_\infty=\sum_A \chi_\infty (a)\otimes_{\mathbb{Z}}F^+(a)_\infty$$ for some $\tilde{c}\in\mathbb{C}$. Consider the map $$\gamma :\mathbb{C}\otimes_{\mathbb{Z}} E(M_f)\rightarrow \mathbb{C},$$ defined by $\gamma(c\otimes a)=c\log a$. This map is clearly $\mathbb{C}$-linear, so $$\gamma(\tilde{c}\sum_A \chi_\infty (a)\otimes_{\mathbb{Z}}C^+(a)_\infty)=\tilde{c}\gamma(\sum_A \chi_\infty (a)\otimes_{\mathbb{Z}}C^+(a)_\infty).$$ Applying $\gamma$ to both sides of $\tilde{c}\sum_A \chi_\infty (a)\otimes_{\mathbb{Z}}C^+(a)_\infty=\sum_A \chi_\infty (a)\otimes_{\mathbb{Z}}F^+(a)_\infty$ yields that $$\tilde{c}=-3f B_{1,\chi_{\infty}\epsilon}.$$ In particular, note that $E(M_f)\subset \overline{\mathbb{Q}}$. We can actually say that
\begin{prop}
\label{p2.11}
As elements of $\overline{\mathbb{Q}}\otimes_{\mathbb{Z}} E(M_f)$, we have $$-3fB_{1,\chi_{\infty}\epsilon}\sum_A \chi_\infty (a)\otimes_{\mathbb{Z}}C^+(a)_\infty=\sum_A \chi_\infty (a)\otimes_{\mathbb{Z}}F^+(a)_\infty.$$
\end{prop}
Now take some $\varphi: \overline{\mathbb{Q}}\xhookrightarrow{}\mathbb{C}_p$. Applying $(1\otimes_\mathbb{Z}\log_p)\circ\varphi $ to both sides, we find the following equality in $\mathbb{C}_p$: $$-3fB_{1,\chi\epsilon}\sum_A \chi (a)\log_p C^+(a)_p=\sum_A \chi (a)\log_p F^+(a)_p.$$
Now consider the explicit formulas provided by \cite[p.~93]{gross1980factorization}:
\begin{align*}
L_p(0,\chi_K)&=-\frac{1}{3f}g(\chi^{-1})\sum_A \chi(a)\log_p F^+(a)_p,\\
L_p(0,\chi \epsilon \omega) &= -B_{1,\chi\epsilon},\\
L_p(1,\chi^{-1})&=-g(\chi^{-1})\sum_A\chi(a)\log_p C^+(a)_p.\\
\end{align*}
Putting these together yields the $p$-adic identity
\begin{prop}
\label{p2.12}
$$L_p(0,\chi_K)=L_p(0,\chi\epsilon\omega)L_p(1,\chi^{-1}).$$
\end{prop}
Now, following \cite[p.~93]{gross1980factorization} there exist measures $\lambda_2,\lambda_3$ such that for any finite even Dirichlet character $\chi$ of conductor $f$, we have 
\begin{align*}
\langle \chi, \lambda_2\rangle &= L_p(0,\chi\epsilon\omega),\\
\langle \chi,\lambda_3\rangle &=L_p(1,\chi^{-1}).\\
\end{align*}
Now define a measure $\lambda_1$ given by $$\langle \chi, \lambda_1\rangle = L_p(0,\chi_K).$$ Then we have the equality $$\langle \chi, \lambda_1\rangle = \langle \chi,\lambda_2\rangle \langle \chi,\lambda_3\rangle$$ for all finite even Dirichlet characters $\chi$ with conductor $f$. Combining with Theorem~\ref{t2.6}, we have the following factorization:
\begin{theorem}
\label{t2.13}
For any finite even Dirichlet character $\chi$, with $\langle \chi, \lambda_1\rangle = L_p(0,\chi_K)$, $\langle \chi, \lambda_2\rangle = L_p(0,\chi\epsilon\omega)$, and $\langle \chi,\lambda_3 \rangle=L_p(1,\chi^{-1})$, we have that $\lambda_1=\lambda_2\cdot\lambda_3$.
\end{theorem}

\section{Assumptions and conventions}
\label{sec:conventions}
For the remainder of the article, we fix several assumptions. We will restate them throughout the article, but organize them here for convenience. We will let $K$ denote an imaginary quadratic field and $\Delta_K$ the discriminant of $K$. We will let $E/F$ denote an elliptic curve over a number field $F$ (usually either $K$ or $\mathbb{Q}$) with conductor $\text{cond}(E)=N$. Let $f(q)$ be the modular form associated to $E$. We assume that $E$ will be residually reducible modulo $3$:
\begin{assumption}[Residually reducible]
\label{assump:residuallyreducible}
All elliptic curves $E$ will be residually reducible modulo $3$. In other words, the $3$-adic Galois representation $\rho_3:\gal(\overline{F}/F)\rightarrow \text{Aut}(T_3(E))\cong GL_2(\mathbb{Z}_3)$ reduced modulo $3$ to $\rho_3:\gal(\overline{F}/F)\rightarrow GL_2(\mathbb{F}_3)$ is reducible.
\end{assumption}

We will also require that $E$ satisfies the Heegner hypothesis relative to $K$ in many situations, as found in \cite[p.~8]{bertolini2014p}:
\begin{assumption}[Heegner hypothesis]
\label{assump:heegner}
For every prime $\ell|N$, then $\Delta_K$ is a quadratic residue modulo $\ell$.
\end{assumption}

We will focus a great deal of attention to quadratic twists of $E/F$. Let $E$ be an elliptic curve given by $y^2=x^3+ax+b$. Then for $D'\in F$ such that $F(\sqrt{D'})\supsetneq F$, the \textbf{quadratic twist} of $E/F$ by $D'$ is given by $D'y^2=x^3+ax+b$, and denoted $E^{(D')}$ with modular form $f^{(D')}(q)$. We will primarily focus on $D'\in \mathbb{Z}$. We will later see that when $E/F$ is residually reducible, then $f^{(D')}(q)\equiv E_2^{\chi_D,\chi_D}(q)\pmod{3}$ for some Eisenstein series $E$ and integer $D$, and therefore we will denote $E_D\coloneqq E^{(D')}$.

In a similar manner, we will denote $E_{d,3}$ the \textbf{cubic twist} of $E/F$ by $d$, where the cubic twist is given by $y^2=x^3+c\mapsto y^2=x^3+dc$.

Finally, we will denote by Assumption~\ref{assump:*} the following series of assumptions on $(N,D)$.
\begin{assumption}
\label{assump:*}
We make the following assumptions:
\begin{itemize}
    \item For all primes $\ell>3$, if $v_\ell(N)=1$, then $\ell\equiv 2\pmod{3}$,
    \item $\gcd (N,D)=1$,
    \item $2\nmid ND$,
    \item $v_3(N)\ne 1$.
\end{itemize}
\end{assumption}

\section{Congruences modulo $p$}
\label{sec:congruences}
\subsection{Congruences of $L$-series and Eisenstein series}
Let $$f(q)=\sum_{n\ge 0}a_nq^n$$ be the modular form attached to an elliptic curve $E$. Let the $3$-adic Galois representation be $\rho_E$; we will assume that $\rho_E$ is always \textit{residually reducible} modulo $3$. Let $$E_{2}^{\lambda,\psi}(q)=L(1-k,\chi)+\sum_n \sigma^{\lambda,\psi}(n)q^n$$ be an Eisenstein series, where $$\sigma^{\lambda,\psi}(n)=\sum_{d|n}\lambda(n/d)\psi(d)d.$$

\begin{prop}
\label{p4.1}
The Galois representation of $E_2^{\psi,\overline{\psi}}$ is isomorphic up to semisimplification to $\psi\oplus \overline{\psi}\chi$ where $\chi$ is the cyclotomic character.
\end{prop}
\begin{proof}
The Brauer-Nesbitt theorem implies that up to semisimplication, $\rho_{E_2}$ is determined by its characteristic polynomial, or equivalently, trace and determinant. Furthermore, the Cebotarev density function implies that the Frobenius elements are dense in the Galois group. Since $\rho_{E_2}$ is a continuous function, it suffices to check that trace and determinant match on the Frobenius elements $\ell$ for each prime. We have $$\text{tr}(\ell)=[q^\ell]E_2^{\psi,\overline{\psi}}(q)=\sigma^{\psi,\overline{\psi}}(\ell)=\psi(\ell)+\overline{\psi}(\ell)\ell,$$ which confirms that the trace function matches. The determinant yields $$\det(\ell)=\ell=\psi(\ell)\overline{\psi}(\ell)\chi(\ell),$$ and both functions match.
\end{proof}

\begin{prop}
\label{p4.2}
Suppose $\rho_E$ is residually reducible, i.e. the representation $\bmod\hspace{1mm} 3$ is isomorphic to $\chi_1\oplus \chi_2$ up to semisimplification. Then $\rho_E\cong \rho_{E_2^{\chi_M,\chi_M}}\pmod{3}$, where $E_2^{\chi_M,\chi_M}$ is the Eisenstein series with $\chi_M$ a quadratic character, and $\chi_1=\chi_M$ and $\chi_2=\chi_M\chi$ for $\chi$ the cyclotomic character. Furthermore, $f(q)\equiv E_2^{\chi_M,\chi_M}(q)\pmod{3}$.
\end{prop}
\begin{proof}
Since $\mathbb{F}_3^\times\cong \{\pm 1\}$, it follows that $\chi_1,\chi_2$ are quadratic characters. By Brauer-Nesbitt, it suffices to check that the trace and determinant functions agree on Frobenius elements, which are dense in the Galois group by the Cebotarev density theorem. Checking the determinant function, we have that $$\chi(\ell)=\ell=\det(\ell)=\chi_1(\ell)\chi_2(\ell).$$ Thus $$\chi_2=\chi_1^{-1}\chi=\chi_1\chi.$$ Letting $\chi_1=\chi_M$ for some quadratic character modulo $M$, we have that $$\chi_1=\chi_M,\hspace{2mm}\chi_2=\chi_M\chi.$$ Now, the trace functions yield that $$a_{\ell}=\chi_M(\ell)+\chi_M(\ell)\chi(\ell)=\chi_M(\ell)+\chi_M(\ell)\ell=\sum_{d|\ell}\chi_M(\ell/d)\chi_M(\ell)d=\sigma^{\chi_M,\chi_M}(\ell)=[q^\ell]E_2^{\chi_M,\chi_M}(q).$$ Since the coefficients of the two modular forms agree on prime indices, they agree on all nonconstant terms. Thus $f(q)-E_2^{\chi_M,\chi_M}(q)\equiv c\pmod{3}$, where the left hand side is a modular form of weight $2$, and thus the right hand side must also be a modular form of weight $2$. By \cite{serre1973formes}, $c=0$, and we have that $f(q)\equiv E_2^{\chi_M,\chi_M}(q)\pmod{3}$.
\end{proof}
\subsection{Congruences of quadratic twists}
Consider some arbitrary squarefree $D'\in \mathbb{Z}$. We will study the \textbf{quadratic twist} of $f$ by $D'$ and write it as $f^{(D')}(q)$, with the elliptic curve $E\coloneqq y^2=x^3+ax+b$ becoming $E^{(D')}\coloneqq D'y^2=x^3+ax+b$. We will assume that $\rho_E$ is residually reducible. By Proposition~\ref{p4.2}, we have $f(q)\equiv E_2^{\chi_M,\chi_M}(q)\pmod{3}$ for some quadratic character $\chi_M$. Since $f_{D'}(q)=\sum_{n\ge 0}\chi_{D'}(n)a_nq^n$, it follows that $f_{D'}(q)\equiv E_2^{\chi_M\chi_{D'},\chi_M\chi_{D'}}(q)\pmod{3}$. Since $\chi_{D'}$ is again a quadratic character, we may write $\chi_D=\chi_M\chi_{D'}$ for some squarefree $D\in\mathbb{Z}$. Thus every quadratic twist of an elliptic curve whose Galois representation is residually reducible $\bmod\hspace{1mm}3$ is congruent to $E_2^{\chi_D,\chi_D}(q)$ modulo $3$, where $\chi_D$ is some quadratic character. From now on, we will write the quadratic twist of the elliptic curve $E^{(D')}$ as $E_D$ and the associated modular form as $f_D$, where $\chi_D=\chi_M\chi_{D'}$.

\subsection{Stabilizations}
\label{subsec:stabilize}
We follow \cite{bertolini2013generalized} and describe the $p$-depletions/stabilizations. Let $f(q)=\sum_{n}a_nq^n$ be the modular form of an elliptic curve $E/F$. Then the $p$-depletion is given by $$f^\flat (q)=\sum_{p\nmid n}a_n q^n.$$ Suppose $p$ is a good prime; then $f^\flat =f|VU-UV$, where $U$ and $V$ are given by \cite[p.~1085]{bertolini2013generalized}.

Let $N$ be the conductor of $E/F$. If a prime $\ell^2|N$, then $a_\ell=0$, hence there is no need to change the value through $\ell$-depletion. If a prime $\ell|N$ with $v_{\ell}(N)=1$, then $a_\ell=\pm 1$, so we may use $1\mp V_\ell$ instead of $VU-UV$. In particular, $1-a_\ell V=1-T_\ell V$ suffices.

We assume the \textbf{Heegner hypothesis} \cite[p.~8]{bertolini2014p}, Assumption~\ref{assump:heegner}: for every $\ell|N$, the ideal $(\ell)$ splits in $\mathcal{O}_F$ as $\mathfrak{l}\cdot \overline{\mathfrak{l}}$.

\begin{prop}
\label{p5.1}
The Euler factor at $\ell$ of $S_\chi^{\flat_\ell}$ for some $\ell|N$ is $1-\chi^{-1}(\overline{\mathfrak{l}})$.
\end{prop} 
\begin{proof}
Following \cite[p.~1135]{bertolini2013generalized}, we set $$S_\chi^{\flat_\ell}=\sum_{[\mathfrak{a}]}\chi_j^{-1}(\mathfrak{a})\cdot \theta^j f^{\flat_\ell}(\mathfrak{a}*(A_0,t_0,\omega_0)).$$ Let $a_\ell=\pm 1$. Then
\begin{align*}
\theta^j f^{\flat_\ell}(\mathfrak{a}*(A_0,t_0,\omega_0))&=\{\theta^j f | (1\mp T_\ell V)\}(\mathfrak{a}*(A_0,t_0,\omega_0)),\\
&=\theta^j f(\mathfrak{a}*(A_0,t_0,\omega_0))\mp \ell^j a_\ell \theta^j f(\overline{\mathfrak{l}}^{-1}\mathfrak{a}*(A_0,t_0,\omega_0)),\\
\end{align*} so $$S_{\chi}^{\flat_{\ell}}=(1\mp a_\ell\chi^{-1}(\overline{\mathfrak{l}}))S_\chi=(1-\chi^{-1}(\overline{\mathfrak{l}}))S_\chi.$$
\end{proof}
\begin{corollary}
\label{c5.1.1}
Assuming the Heegner hypothesis, the Euler factor at $\ell\ne p$ does not vanish modulo $p$ when $\ell\ne 1\pmod{p}$.
\end{corollary}
\begin{proof}
Since $\chi(\overline{\mathfrak{l}})=\ell$, it suffices to have $1-\ell^{-1}\not\equiv 0\pmod{p}\iff \ell\not\equiv 1\pmod{p}$.
\end{proof}

If $f_1(q)$ and $f_2(q)$ agree on all coefficients $[q^n]f_i(q)$ whenever $p_i\nmid n$ for all $p_i\in X$ where $X$ is a finite set of primes, then we may take the $p_i$-depletions to force them to be equal. In particular, $$f_1^{\flat_{X}}(q)=f_2^{\flat_X}(q),$$ where $f_i^{\flat_X}(q)$ indicates the modular form $f_i(q)$ after $p_j$ depletions for each $p_j$.

\section{Congruence of modular forms}
\label{sec:modularforms}
\subsection{Varying $K$}
\label{subsec:varyK}
We follow the discussion from section $3$. For the remainder of this section, we set $p=3$. Suppose we have an $L$-series attached to an elliptic curve $E$ whose Fourier expansion is $f(q)$ whose Galois representation $\rho_E$ modulo $3$ is residually reducible. Then proposition $15$ implies that $\rho_E\cong \rho_{E_2^{\chi_M,\chi_M}}\pmod{3}$, the Galois representation modulo $3$ of the Eisenstein series $E_2^{\chi_M,\chi_M}$. Now consider the quadratic twist by $D'$, so that $f^{(D')}(q)=\sum \chi_{D'}(n)a_nq^n$, where $\chi_{D'}$ is the Kronecker character. It follows that $\rho_{E^{(D')}}\cong \rho_{E_2^{\chi_D,\chi_D}}\pmod{3}$ where $\chi_D=\chi_{D'}\chi_M$ is some quadratic character. Now denote $E_D\coloneqq E^{(D')}$, and $f_D(q)\coloneqq f^{(D')}(q)$, so that we parametrize the twists by the corresponding Eisenstein series. In particular, $[q^\ell]f(q)\equiv [q^\ell]E_2^{\chi_D,\chi_D}(q)\pmod{3}$ for all $\ell\nmid \text{cond}(E_D)=ND^2$, where $N=\text{cond}(E)$. We thus have that $f_D(q)$ and $E_2^{\chi_D,\chi_D}(q)$ are congruent modulo $3$ everywhere except possibly at indices divisible by some bad prime $\ell$.

Let $N=\text{cond}(E)$. Assume $p\nmid ND$. If $\gcd(N,D)=1$, then $\text{cond}(E_D)=\text{level}(f_D)=ND^2$. On the other hand, $D=\text{cond}(\chi_D)$, so $\text{level}(E_{2}^{\chi_D,\chi_D})=D^2$. We need only stabilize ($\ell$-deplete) at primes $\ell$ such that $v_\ell(ND^2)=1$. Due to Corollary~\ref{c5.1.1}, we will require all such $\ell$ to satisfy $\ell\not\equiv 1\pmod{3}$ for the rest of the paper. This is noted in the section on assumptions. For each of these $\ell$, we have $\ell \nmid D$, so $\ell$ is a good prime for $E_2^{\chi_D,\chi_D}$. On the other hand, for $\ell|D^2$, it immediately follows that $\ell^2|D^2$ so $[q^\ell] E_2^{\chi_D,\chi_D}=0$, and since $\ell^2|ND^2$, then $[q^\ell]f_D=0$. Hence there is no need to $\ell$-deplete $E_{k}^{\chi_D,\chi_D}(q)$ at such primes (it is already zero), and we only need to consider the primes $\ell$ for which $v_\ell(N)=1$.

Denote this set by $X$. Then $X$ is a set of bad primes for $f_D(q)$, but good primes for $E_{2}^{\chi_D,\chi_D}(q)$. Take $Z$ to be the product of $\ell\in X$.

Note that $S_{\chi}^{\flat_\ell}(f)\equiv S_{\chi}(f^{[\ell]})\pmod{p}$ and $f_D(q)\equiv E_2^{\chi_D,\chi_D}(q)\pmod{p}$, so by the $q$-expansion principle, we have 
\begin{prop}
\label{p6.1}
For infinity types $\chi$ of type $(k+j,-j)$ with $j\ge 0$, we have $S_{\chi}^{\flat}(f_D^{[pZ]})\equiv S_{\chi}^{\flat}(E_{2}^{\chi_D,\chi_D [pZ]})\pmod{p}$.
\end{prop} 
Setting $k=2$ and taking the limit of $j_m=p^m-1$ as $m\rightarrow \infty$ gives, by continuity,
\begin{prop}
\label{p6.2}
For $\mathbb{N}_K$ the norm character of type $(1,1)$, we have $$S_{\mathbb{N}_K\chi_D}^{\flat}(f_D^{[pZ]})\equiv S_{\mathbb{N}_K\chi_D}^{\flat}(E_{2}^{\chi_D,\chi_D [pZ]})\pmod{p}.$$
\end{prop}

In fact, more generally:
\begin{prop}
\label{p6.3}
Suppose we have two modular forms $f$ and $g$ with Galois representations $\rho_f$ and $\rho_g$ such that $\rho_f\equiv \rho_g\pmod{p}$. Then $S_{\chi}^\flat(f^{[N]})\equiv S_{\chi}^{\flat}(g^{[N]})\pmod{p}$, where $N$ is such that $[q^\ell]f\equiv [q^\ell]g\pmod{p}$ whenever $\ell\nmid N$, and $\chi$ is type $(k+j,-j)$ for $j\ge 0$. Furthermore, this is true for $\chi=\mathbb{N}_K$, the norm character of infinity type $(1,1)$.
\end{prop}
\begin{proof}
The stabilization at $\ell$ yields $[q^\ell]f^{[\ell]}=0$, and thus after stabilizing at all $\ell|N$, it follows that $$\ell|N\implies [q^\ell]f^{[N]}\equiv [q^\ell]g^{[N]}\equiv 0\pmod {p}$$ and by hypothesis, they are already congruent modulo $p$ when $\ell\nmid N$. The $q$-expansion principle then implies that $$S_{\chi}^{\flat}(f^{[N]})\equiv S_{\chi}^{\flat}(g^{[N]})\pmod{p}.$$

To see that this holds for $\chi=\mathbb{N}_K$, take $\chi$ of type $(2+p^m-1,1-p^m)$ as $m\rightarrow \infty$. Since $S_{\chi}^{\flat}$ is continuous in $\chi$, it follows that the limit is $(2-1,1)=(1,1)$ and the congruence for $\mathbb{N}_K$ holds.
\end{proof}

Define $L_{p,\alpha}(w,\chi)=S_\chi^{\flat}(w)$ as in \cite[Definition~8.8]{krizsupersingular}; we'll write $L_p(w,\chi)$ as shorthand. 

\begin{prop}
\label{p6.4}
If $S_\chi^\flat (E_2^{\chi_D,\chi_D [pZ]})\not\equiv 0\pmod{p}$, then $E_D/K$ has rank $1$.
\end{prop}
\begin{proof}
We have $S_{\mathbb{N}_K\chi_D}^\flat (E_{2}^{\chi_D,\chi_D[pZ]})=L_{p}(E_2^{\chi_D,\chi_D[pZ]},\mathbb{N}_K\chi_D)$ and $S_{\mathbb{N}_K\chi_D}^\flat (f_D^{[pZ]})=L_{p}(f_D^{[pZ]},\mathbb{N}_K\chi_D)$. From Proposition~\ref{p6.2}, $L_p(E_2^{\chi_D,\chi_D[pZ]},\mathbb{N}_K\chi_D)\equiv L_p(f_D^{[pZ]},\mathbb{N}_K\chi_D)\pmod{p}$. By \cite[Theorem~9.10]{krizsupersingular}, $L_{p}(f_D,\mathbb{N}_K\chi_D) = \Omega(A,t)  \Xi_p(f_D,\mathbb{N}_K\chi_D) \log_{E_D}(P_K)
$, where $P_K$ is a Heegner point. Now suppose $S_{\chi}^{\flat}(E_{2}^{\chi_D,\chi_D [pZ]})\not\equiv 0\pmod{p}$. By Corollary~\ref{c5.1.1}, due to the Heegner hypothesis, none of the Euler factors vanish, and thus $L_{p}(w,\mathbb{N}_K\chi_D)\not\equiv 0\pmod{p}\implies  L_{p}(w,\mathbb{N}_K\chi_D)\ne 0$. It follows that $\log_{E_D}(P_K)\ne 0$, and hence $P_K$ is not a torsion point, and it follows that $E_D/K$ has positive rank. A theorem due to Kolyvagin \cite{kolyvagin1989finiteness} (for example, see \cite[Theorem~2.9]{darmon2006heegner}) implies that in fact $E_D/K$ has rank exactly $1$.
\end{proof}

Adopting the notation from \cite[Theorem~9.11]{krizsupersingular}, for $D>0$ we have that \begin{align*}
    S_{\mathbb{N}\chi_D} (E_{2}^{\chi_D,\chi_D})&=L_{p,\alpha}(0,(\chi_{D})_K),\\
    &=\Omega(A,t)\frac{\Xi_p(0,(\chi_D)_K)}{g(\chi_D)}\sum_{\mathfrak{a}\in\mathcal{C}\ell(\mathcal{O}_K)}(\chi^{-1}\mathbb{N}_K)(\mathfrak{a})\sum_{a=1}^{N-1}\chi_D^{-1}(a)\log_p g_a(\mathfrak{a}\star (A,t)).
\end{align*}

By Theorem~\ref{t2.13}, this sum is equal to $\mathcal{L}_p(0,\chi_D\chi_K\omega)\mathcal{L}_p(1,\chi_D)$, where $\mathcal{L}_p$ is the Katz $p$-adic $L$-function. By \cite[Theorem~5.11]{washington1997introduction}, we have that $$\mathcal{L}_p(0,\chi_D\chi_K\omega)\mathcal{L}_p(1,\chi_D)=-B_{1,\chi_D\chi_K}\mathcal{L}_p(1,\chi_D)\equiv B_{1,\chi_D\chi_K}B_{1,\chi_D\omega^{-1}}\pmod{p},$$ so we conclude that $$S_{\mathbb{N}\chi_D} (E_{2}^{\chi_D,\chi_D})\equiv B_{1,\chi_D\chi_K}B_{1,\chi_D\omega^{-1}}\pmod{p}.$$

If $D<0$ then $\chi_D$ is odd, so $$\Omega(A,t)\frac{\Xi_p(0,(\chi_D)_K)}{g(\chi_D)}\sum_{\mathfrak{a}\in\mathcal{C}\ell(\mathcal{O}_K)}(\chi^{-1}\mathbb{N}_K)(\mathfrak{a})\sum_{a=1}^{N-1}\chi_D^{-1}(a)\log_p g_a(\mathfrak{a}\star (A,t))=\mathcal{L}_p(0,\chi_D\omega)\mathcal{L}_p(1,\chi_D\chi_K)$$ instead. By \cite[Theorem~5.11]{washington1997introduction} we have that $$\mathcal{L}_p(0,\chi_D\omega)\mathcal{L}_p(1,\chi_D\chi_K)=-B_{1,\chi_D}\mathcal{L}_p(1,\chi_D\chi_K)\equiv B_{1,\chi_D}B_{1,\chi_D\chi_K\omega^{-1}}\pmod{p}.$$ This implies that for $D<0$, we have $$S_{\mathbb{N}\chi_D}(E_2^{\chi_D,\chi_D})\equiv B_{1,\chi_D}B_{1,\chi_D\chi_K\omega^{-1}}\pmod{p}.$$

Hence 
$$S_{\mathbb{N}\chi_D}(E_2^{\chi_D,\chi_D})\equiv 
\begin{cases}
B_{1,\chi_D\omega^{-1}}B_{1,\chi_D\chi_K}\pmod{p} & D>0,\\
B_{1,\chi_D\chi_K\omega^{-1}}B_{1,\chi_D}\pmod{p} & D<0.\\

\end{cases}
$$

For $D>0$, this turns out to be $h_{\mathbb{Q}(\sqrt{-3D})}h_{\mathbb{Q}(\sqrt{D\Delta_K})}$. For $D<0$, this turns out to be $h_{\mathbb{Q}(\sqrt{-3DD_K})}h_{\mathbb{Q}(\sqrt{D})}$.

\begin{prop}
\label{p6.5}
$$S_{\mathbb{N}\chi_D}(E_2^{\chi_D,\chi_D})\equiv 
\begin{cases}
h_{\mathbb{Q}(\sqrt{-3D})}h_{\mathbb{Q}(\sqrt{D\Delta_K})}\pmod{p} & D>0,\\
h_{\mathbb{Q}(\sqrt{-3D\Delta_K})}h_{\mathbb{Q}(\sqrt{D})}\pmod{p} & D<0.\\

\end{cases}
$$
\end{prop}

We now turn to calculating the proportion of $\Delta_K$ such that $p=3\nmid h_{\mathbb{Q}(\sqrt{-3D})}h_{\mathbb{Q}(\sqrt{D\Delta_K})}$; we will address the other case shortly after.

We assume the Heegner hypothesis. This requires that $\Delta_K$ is a quadratic residue modulo all primes dividing $\text{cond}(E_D)$, except for $3$. Letting $N=\text{cond}(E)$ and $\gcd(N,D)=1$ such that $2\nmid ND$, then $\text{cond}(E_D)=ND^2$. Furthermore, $v_3(N)\ne 1$ (due to the conditions provided by Nakagawa-Horie in \cite[p.~21]{nakagawa1988elliptic} or \cite[Lemma~2.2]{byeon2004class}). Recall that Assumption~\ref{assump:*} denotes the follow conditions on $(N,D)$:
\begin{itemize}
    \item For all primes $\ell>3$, if $v_\ell(N)=1$, then $\ell\equiv 2\pmod{3}$,
    \item $\gcd (N,D)=1$,
    \item $2\nmid ND$,
    \item $v_3(N)\ne 1$.
\end{itemize}
\begin{theorem}
\label{t6.6}
For fixed $N,D$ with $D>0$ satisfying Assumption~\ref{assump:*} and $3\nmid h_{\mathbb{Q}(\sqrt{-3D})}$, the proportion of $\Delta_K$ such that $3\nmid h_{\mathbb{Q}(\sqrt{D\Delta_K})}$ is at least $2^{-1-\omega(ND/3^{v_3(ND)})}$.
\end{theorem}
\begin{proof}
For each $3\ne \ell|ND$, the proportion of $\Delta_K$ which are quadratic residues $\bmod \ell$ is $\frac{\ell+1}{2\ell}>\frac{1}{2}$. The number of such primes is $\omega(ND/3^{v_3(ND)})$, so the proportion of such $\Delta_K$ is greater than $\frac{1}{2^{\omega(ND/3^{v_3(ND)})}}$. Of this set $X$, \cite[Lemma~2.2]{byeon2004class} implies that $$\frac{|\{\Delta_K\in X\text{ and }3\nmid h_{\mathbb{Q}(\sqrt{D\Delta_K})}\}|}{|X|}\ge \frac{1}{2}.$$ Hence for a fixed $D$, the proportion of $\Delta_K$ which are quadratic residues modulo all $\ell|D$ and $3\nmid h_{\mathbb{Q}(\sqrt{D\Delta_K})}$ is at least $2^{-1-\omega(ND/3^{v_3(ND)})}$.
\end{proof}
We have the immediate
\begin{corollary}
\label{c6.6.1}
For fixed $N=\text{cond}(E)$ and $D>0$ satisfying Assumption~\ref{assump:*} and $3\nmid h_{\mathbb{Q}(\sqrt{-3D})}$, the proportion of imaginary quadratic fields $K$ which admit a quadratic twist of $E$ by the fixed $D$ is positive.
\end{corollary}
We also immediately obtain information about the rank of $E_D$ over $K$.
\begin{corollary}
\label{c6.6.2}
For fixed $N,D$ with $D>0$ satisfying Assumption~\ref{assump:*} and $3\nmid h_{\mathbb{Q}(\sqrt{-3D})}$, the proportion of imaginary quadratic fields $K$ such that $E_D/K$ has rank $1$ is at least $2^{-1-\omega(ND/3^{v_3(ND)})}$.
\end{corollary}
\begin{proof}
Theorem~\ref{t6.6} implies that the proportion of $K$ with $3\nmid h_{\mathbb{Q}(\sqrt{-3D})}h_{\mathbb{Q}(\sqrt{D\Delta_K})}$ is at least $2^{-1-\omega(ND/3^{v_3(ND)})}$. Combining with the fact that $$S_{\mathbb{N}_{\chi_D}}(E_2^{\chi_D,\chi_D})\equiv h_{\mathbb{Q}(\sqrt{-3D})}h_{\mathbb{Q}(\sqrt{D\Delta_K})}\pmod{3},$$ we have that $3\nmid S_{\mathbb{N}_{\chi_D}}(E_2^{\chi_D,\chi_D})$. Now applying Proposition~\ref{p6.4}, we find that every such $K$ also satisfies that $E_D/K$ has rank $1$.
\end{proof}

We also address the $D<0$ case. Once again, let $N=\text{cond}(E)$ and $\gcd(N,D)=1$ with $2\nmid ND$, and $v_3(N)\ne 1$.
\begin{theorem}
\label{t6.7}
For fixed $N,D$ with $D<0$ satisfying Assumption~\ref{assump:*} and $3\nmid h_{\mathbb{Q}(\sqrt{D})}$, the proportion of $\Delta_K$ such that $3\nmid h_{\mathbb{Q}(\sqrt{-3D\Delta_K})}$ is at least $2^{-1-\omega(ND/3^{v_3(ND)})}$.
\end{theorem}
\begin{proof}
For each $3\ne \ell|ND$, the proportion of $D_K$ which are quadratic residues $\bmod{\ell}$ is $\frac{\ell+1}{2\ell}>\frac{1}{2}$. The number of such primes is $\omega(ND/3^{v_3(ND)})$, so the proportion of such $\Delta_K$ is greater than $2^{-\omega(ND/3^{v_3(ND)})}$. This set $X$ is given by a system of congruence conditions, modulo all $3\ne \ell|ND$. Of this set $X$, \cite[Lemma~2.2]{byeon2004class} implies that $$\frac{\left|\Delta_K\in X\text{ and }3\nmid h_{\mathbb{Q}(\sqrt{-3D\Delta_K})}\right|}{|X|}\ge \frac{1}{2}.$$ Hence for fixed $D$, the proportion of $\Delta_K$ which satisfy the Heegner hypothesis and $e\nmid h_{\mathbb{Q}(\sqrt{-3D\Delta_K})}$ is at least $2^{-1-\omega(ND/3^{v_3(ND)})}$.
\end{proof}
Once again, we find two corollaries.
\begin{corollary}
For fixed $N=\text{cond}(E)$ and $D<0$ satisfying Assumption~\ref{assump:*} and $3\nmid h_{\mathbb{Q}(\sqrt{D})}$, the proportion of imaginary quadratic fields $K$ which admit a quadratic twist of $E$ by the fixed $D$ is positive.
\end{corollary}
\begin{corollary}
For fixed $N,D$ with $D<0$ satisfying Assumption~\ref{assump:*} and $3\nmid h_{\mathbb{Q}(\sqrt{D})}$, the proportion of imaginary quadratic fields $K$ such that $E_D/K$ has rank $1$ is at least $2^{-1-\omega(ND/3^{v_3(ND)})}$.
\end{corollary}
\begin{proof}
Theorem~\ref{t6.7} implies that the proportion of $K$ with $3\nmid h_{\mathbb{Q}(\sqrt{D})}h_{\mathbb{Q}(\sqrt{-3D\Delta_K})}$ is at least $2^{-1-\omega(ND/3^{v_3(ND)})}$. Combining with the fact that $$S_{\mathbb{N}_{\chi_D}}(E_2^{\chi_D,\chi_D})\equiv h_{\mathbb{Q}(\sqrt{D})}h_{\mathbb{Q}(\sqrt{-3D\Delta_K})}\pmod{3},$$ we have that $3\nmid S_{\mathbb{N}_{\chi_D}}(E_2^{\chi_D,\chi_D})$. Now applying Proposition~\ref{p6.4}, we find that every such $K$ also satisfies that $E_D/K$ has rank $1$.
\end{proof}

\subsection{Varying $D$}
\label{subsec:varyD}
In this paper we will usually fix $D$ and vary $K$. Let us now fix $N$ and $K$ and vary $D$. The first result we have is considering the proportion of $D$ which satisfy the Heegner hypothesis.

\begin{theorem}
\label{t6.8}
For fixed $N,K$, the number of $0<D<X$ satisfying the Heegner hypothesis is asymptotic to $X/(\log X)^{1/2}$.
\end{theorem}
\begin{proof}
For each $3\ne \ell|ND$, we need $\Delta_K$ to be a quadratic residue $\bmod$ $\ell$. Let $\Delta_K=-2^{e}\prod_{i=1}^{M} p_i$ where the $p_i$ are distinct odd primes and $e\in\{0,2\}$. We need $$\left(\frac{-1}{\ell}\right)\prod_i \left(\frac{p_i}{\ell}\right)=1,$$ since the factor of $2$ is always a square. Quadratic reciprocity implies that $\left(\frac{p_i}{\ell}\right)=(-1)^{\frac{p_i-1}{2}\cdot \frac{\ell-1}{2}}\left(\frac{\ell}{p_i}\right)$ for each $p_i$. In particular, when $\ell$ is fixed, the sign depends only on $p_i$, and is also fixed. Then $\left(\frac{\Delta_K}{\ell}\right)=1$ is equivalent to the condition that an even number of the $(-1)^{\frac{p_i-1}{2}\cdot \frac{\ell-1}{2}}\left(\frac{\ell}{p_i}\right)$ are $-1$, and thus in the product $$\left(\frac{\Delta_K}{\ell}\right)=\prod_i (-1)^{\frac{p_i-1}{2}\cdot \frac{\ell-1}{2}}\left(\frac{\ell}{p_i}\right),$$ it suffices to allow anything in the first $M-1$ indices, and the last index is determined in order to yield a product of $1$. Note that if $\ell=p_i$ for any $i$, then it is always a quadratic residue. Thus at most $\frac{p_i-1}{2p_i}$ of $\ell$ have $\Delta_K$ not a quadratic residue $\pmod{\ell}$, and it follows that the proportion of $\ell$ with $\Delta_K$ a quadratic residue $\pmod{\ell}$ is at least $\frac{1}{2}$. Even stronger, $\Delta_K$ is a quadratic residue $\pmod{\ell}$ whenever $\bar{\ell}\in S\subset \mathbb{Z}/\Delta_K'\mathbb{Z}$ where $\Delta_K'=\Delta_K$ or $4\Delta_K$, and $|S|>|\Delta_K'|/2$. It follows that $D$ can only be constructed from such primes.

From the Wiener-Ikehara Tauberian theorem \cite[Theorem~2.4]{serre1974divisibilite}, we find that the proportion of $D<X$ such that $D$ is constructed from this set of primes is asymptotic to $X/(\log X)^{1/2}$.

\end{proof}

Suppose we let $D$ vary and count the proportion of pairs $(D,K)$ (equivalently pairs $(D,\Delta_K)$) such that $E_D/K$ has rank $1$. Corollary~\ref{c6.6.1} and Corollary~\ref{c6.6.2} imply that the proportion of $\Delta_K$ for fixed $D$ depends only on $D$, and in particular, on the number of prime factors of $D$. Thus it suffices to consider only positive $D$, as the sign does not matter. Consider the interval $0<D<X$. Then we seek to measure $$F(X)=\frac{1}{X}\sum_{D=1}^{X}2^{-\omega(D)}.$$ By summing over values of $\omega(n)$ instead, we have $$F(X)=\frac{1}{X}\sum_{n=0}^{\log X}2^{-n}\cdot\# \{D\le X| \omega(D)=n\}.$$ The final proportion will be at least $\frac{1}{2}F(X)$. We will now find asymptotic bounds for $F(X)$.

By Erd\"os-Kac \cite{erdos1940gaussian}, for $1\le n\le X$, $\omega(n)$ follows a Gaussian distribution with $\sigma=\sqrt{\log \log X}$ and $\mu=\log \log X$. Therefore we may assume that for sufficiently large $X$, $\omega$ may be approximated by a continuous distribution; we will assume that this continuous distribution is sufficiently accurate and measure $$F(X)\sim T(X)=\int_{0}^{X\cdot (Y-\log\log X)/\sqrt{\log \log X}}2^{-x}\cdot \frac{1}{\sigma \sqrt{2\pi}}e^{-\frac{1}{2}\left(\frac{x-\mu}{\sigma}\right)^2}\,dx,$$ where $Y=\max \{\omega(D)|D\le X\}$. Since for large $X$, it's clear that the upper bound exceeds $2X$, we may take $$S(X)=\int_{0}^{2X}2^{-x}\cdot \frac{1}{\sigma \sqrt{2\pi}}e^{-\frac{1}{2}\left(\frac{x-\mu}{\sigma}\right)^2}\,dx<T(X).$$

Substituting $y=2^{-x}$, we transform the following integral, which is $S(X)$ but extended from $-\infty$ to $\infty$: $$\int_{-\infty}^{\infty}2^{-x}\cdot \frac{1}{\sigma \sqrt{2\pi}}e^{-\frac{1}{2}\left(\frac{x-\mu}{\sigma}\right)^2}\,dx=e^{-(\log 2)\mu +\frac{(\log 2)^2\sigma^2}{2}}>S(X).$$ Let $\kappa = -\frac{(\log 2)^2}{2}+\log 2\approx 0.45$. Then substituting $\mu=\sigma^2=\log\log X$, we have $$S(X)=\frac{1}{(\log X)^\kappa}-A(X)-B(X),$$ where
\begin{align*}
    A(X)&=\int_{-\infty}^{0}2^{-x}\cdot \frac{1}{\sigma \sqrt{2\pi}}e^{-\frac{1}{2}\left(\frac{x-\mu}{\sigma}\right)^2}\,dx,\\
    B(X)&=\int_{2X}^{\infty}2^{-x}\cdot \frac{1}{\sigma \sqrt{2\pi}}e^{-\frac{1}{2}\left(\frac{x-\mu}{\sigma}\right)^2}\,dx.\\
\end{align*} We will now bound $A(X)$ and $B(X)$ (both of which are positive values, since the integrand is strictly positive).

\begin{lemma}
$A(X)<\frac{1}{(\log X)^{1/2}\sqrt{X}}$.
\end{lemma}
\begin{proof}
Take $$A(X)=\int_{0}^{\infty}2^x\cdot \frac{1}{\sigma \sqrt{2\pi}}e^{-\frac{1}{2}\left(\frac{x+\mu}{\sigma}\right)^2}\,dx.$$ Let $N=\mu=\sigma^2=\log\log X$. Then we can write $$A(X)=\frac{1}{N\sqrt{2\pi}}\int_{0}^{\infty}\text{exp}\left(x\log 2-\frac{x^2+2xN+N^2}{2N}\right)\,dx.$$ Now consider that $$x\log 2-\frac{x^2+2xN+N^2}{2N}=-\frac{x^2}{2N}-\frac{N}{2}-\frac{x(1-\log 2)}{N}<-\frac{x^2}{2N}-\frac{N}{2}.$$ As a result, we have that $A(X)<\frac{e^{-N/2}}{N\sqrt{2\pi}}I(X)$, where $I(X)=\int_{0}^{\infty}e^{-\frac{x^2}{2N}}\,dx$. But $I(X)$ can be solved using the well-known Poisson trick, which yields that $I(X)^2=2\pi N\implies I(X)=\sqrt{2\pi N}$, so we have that $A(X)<e^{-N/2}/\sqrt{N}$.
\end{proof}

With $B(X)$, we can calculate it almost exactly. 
\begin{lemma}
$B(X)<e^{-2X^2/\log\log X}$.
\end{lemma}
\begin{proof}
We have $$B(X)=\frac{1}{N\sqrt{2\pi}}\int_{2X}^{\infty}e^{-x\log 2-\frac{(x-N)^2}{2N}}\,dx.$$ The exponent rearranges to $$-\frac{N}{2}(1-(1-\log 2)^2)-\frac{(x-N(1-\log 2))^2}{2N}\implies B(X)=\frac{e^{-\frac{N}{2}(1-(1-\log 2)^2)}}{N\sqrt{2\pi}}J(X),$$ where $$J(X)=\int_{2X}^{\infty}e^{-\frac{(x-(1-\log 2)N)^2}{2N}}\,dx=\int_{2X-(1-\log 2)N}^{\infty}e^{-x^2/2N}\,dx.$$ Using Poisson's trick once again, we find $$J(X)^2<2\pi\int_{2X-(1-\log 2)N}^{\infty}re^{-r^2/2N}\,dr=2\pi Ne^{-\frac{(2X-(1-\log 2)N)^2}{2N}}.$$ As a result, we conclude that $$\log B(X)<-\frac{N}{2}(1-(1-\log 2)^2)-\frac{(2X-(1-\log 2)N)^2}{2N}=-\frac{2X^2}{N}+2X(1-\log 2)-\frac{N}{2}<-\frac{2X^2}{N}.$$
\end{proof}

Putting the above two lemmas together, we conclude that $$\frac{1}{(\log X)^{\kappa}}-\frac{1}{(\log X)^{1/2}\sqrt{X}}-\frac{1}{e^{2X^2/\log\log X}}<S(X)<\frac{1}{(\log X)^{\kappa}},$$ which implies that
\begin{theorem}
\label{t6.11}
Assuming that $\omega$ is approximated by a (continuous) Gaussian distribution sufficiently well, in the set $\mathcal{D}_X=\{(D,\Delta_K)|\hspace{1mm}|D|<X, K \text{ an imaginary quadratic field}\}$, the proportion $P(X)$ of $\mathcal{D}_X$ (for $X\gg 0$) which yield a quadratic twist with rank $1$ over $K$ satisfies $$\frac{1}{2}\left(\frac{1}{(\log X)^{\kappa}}-\frac{1}{(\log X)^{1/2}\sqrt{X}}-\frac{1}{e^{2X^2/\log\log X}}\right)<P(X)<\frac{1}{2(\log X)^\kappa},$$ where $\kappa = \log 2-\frac{(\log 2)^2}{2}\approx 0.45$.
\end{theorem}

\section{Higher twists}
\label{sec:highertwists}
\subsection{Cubic twists}
\label{subsec:cubic}
Let $f$ be the modular form associated to $E$, an elliptic curve over $L=\mathbb{Q}(\sqrt{-3})$. We have that $f=\theta_{\psi}$ (of weight $2$) for some Hecke character $\psi$ of type $(1,0)$, where $$\theta_\psi=\sum_{\mathfrak{a}\subset \mathcal{O}_L}\psi(\mathfrak{a})q^{N(\mathfrak{a})}\in\theta_\psi\in \mathbb{Z}[[q]].$$ Let $\chi_d:\text{Gal}(L(\sqrt[3]{d})/L)\rightarrow \langle\zeta_3\rangle$ be the associated cubic twist character, where $\zeta_3=\frac{-1+\sqrt{-3}}{2}$ is a primitive third root of unity. Let $f_d$ be the modular form associated to $E_{d,3}$, the cubic twist of $E$ by $d$, such that $$f_d=\theta_{\psi \chi_d}=\sum_{\mathfrak{a}\subset\mathcal{O}_L}\psi\chi_d(\mathfrak{a})q^{N(\mathfrak{a})}\in \mathbb{Z}[[q]].$$
\begin{prop}
\label{p7.1}
The modular forms $f$ and $f_d$ are equivalent modulo $3$ at all coefficients except those which are not relatively prime to $Nd$.
\end{prop}
\begin{proof}
For all $n\in\mathbb{Z}$ with $\gcd(n,\text{cond}(\chi_d))=1$, we have $\chi_d(n)\equiv 1\pmod{\zeta_3-1}$. As a result, for all $n$ coprime to $\text{cond}(\chi_d)\cdot \text{cond}(E)=Nd$, we have $$\psi(n)\equiv \psi\chi_d(n)\bmod{(\zeta_3-1)\mathcal{O}_L[[q]]},$$ and thus we have $f\equiv f_d\bmod{(\zeta_3-1)\mathcal{O}_L[[q]]}$ except at the coefficients of $q^n$ for $\gcd(n,Nd)\ne 1$. Since $(\zeta_3-1)\mathcal{O}_L\cap \mathbb{Z}=(3)\mathbb{Z}$ and $f,f_d\in\mathbb{Z}[[q]]$, it follows that $f-f_d\in\mathbb{Z}[[q]]$ and $f-f_d-G(q)\in(\zeta_3-1)\mathcal{O}_L[[q]]$ for some $G(q)\in\mathbb{Z}[[q]]$ with $G(q)=\sum_{\gcd(n,Nd)>1}g_n q^n$. Therefore $f-f_d-G(q)\in\mathbb{Z}[[q]]$, and hence $$f-f_d-G(q)\in(\zeta_3-1)\mathcal{O}_L[[q]]\cap \mathbb{Z}[[q]]=(3)\mathbb{Z}[[q]].$$ As a result, $f-f_d\equiv G(q)\pmod{3}$, and therefore $f\equiv f_d\pmod{3}$ except at coefficients of $q^n$ for $n$ not relatively prime to $Nd$.
\end{proof}

\begin{prop}
\label{p7.2}
The Galois representations of $f$ and $f_d$ are isomorphic$\mod 3$.
\end{prop}
\begin{proof}
Let $\rho_E:\gal(\bar{\mathbb{Q}}/\mathbb{Q})\rightarrow GL(\varprojlim E[n])\cong \prod_{p}GL_2(\mathbb{Z}_p)$ be the Galois representation of $E$. Let $\rho_d$ be the Galois representation of $E_{d,3}$, the cubic twist of $E$ by $d$. Let $N$ be the conductor of $E$, so that $Nd$ is the conductor of $E_{d,3}$. Then Neron-Ogg-Shafarevich implies that $\rho$ and $\rho_d$ are unramified outside of $N$ and $Nd$, respectively. As a result, the Galois representations factor through $\gal(\mathbb{Q}^{(N)}/\mathbb{Q})$ and $\gal(\mathbb{Q}^{(Nd)}/\mathbb{Q}$, respectively, where $\mathbb{Q}^{(n)}$ is the maximal unramified extension of $\mathbb{Q}$ outside of $n$. Now for all primes $\ell\nmid Nd$, the Artin map gives a Frobenius element $\frob_{\ell}$ such that $\text{tr }\frob_{\ell,E} =[q^{\ell}]f$ and $\text{tr }\frob_{\ell,E_{d,3}} [q^{\ell}]f_d$. The prior discussion shows that $$f\equiv f_d\pmod{3}\implies \text{tr }\frob_{\ell,E}\equiv \text{tr }\frob_{\ell,E_{d,3}} \pmod{3}.$$ Furthermore, the Frobenius elements always satisfy $\det \frob_\ell = \ell$. The Brauer-Nesbitt theorem applied to $\rho$ implies that $\rho$ and $\rho_d$ are characterized (up to isomorphism) by their characteristic polynomials, and thus by the trace and determinant functions. We showed that $\rho$ and $\rho_d$ agree on the trace and determinant functions modulo $3$ for all Frobenius elements $\frob_\ell$ with $\ell\nmid Nd$. By the Cebotarev density function, the Frobenius elements have density $1$ in the Galois groups, and therefore all but finitely many of the Frobenius elements are dense in the Galois group. Since trace and determinant are continuous functions, this implies that $\rho$ and $\rho_d$ agree modulo $3$ on trace and determinant on the entire Galois group, and thus they agree modulo $3$ everywhere (by Brauer-Nesbitt). As a result, we find that $\rho\equiv \rho_d\pmod{3}$.
\end{proof}

\begin{prop}
\label{p7.3}
If $S_{\mathbb{N}_K}^{\flat}(f^{[Nd]})\not\equiv 0\pmod{3}$, then $E_{d,3}/K$ has rank $1$.
\end{prop}
\begin{proof}
By Proposition~\ref{p6.2}, we have that $$S_{\mathbb{N}_K}^{\flat}(f^{[Nd]})\equiv S_{\mathbb{N}_K}^{\flat}(f_d^{[Nd]})\pmod {3}.$$ Following \cite[Definition~8.8]{krizsupersingular}, we have $S_{\mathbb{N}_K}^{\flat}(f^{[Nd]})=L_p(f^{[Nd]},\mathbb{N}_K)$ and $S_{\mathbb{N}_K}^{\flat}(f_d^{[Nd]})=L_p(f_d^{[Nd]},\mathbb{N}_K)$. From \cite[Theorem~9.10]{krizsupersingular}, we have $$L_p(f_d^{[Nd]},\mathbb{N}_K)=\omega(A,t)\Xi_p(f_d^{[Nd]},\mathbb{N}_K)\log_{E_{d,3}}(P_K),$$ where $P_K$ is a Heegner point. As a result, if $S_{\mathbb{N}_K}^{\flat}(f^{[Nd]})\not\equiv 0\pmod{3}$, then this implies that $$S_{\mathbb{N}_K}^{\flat}(f_d^{[Nd]})\equiv S_{\mathbb{N}_K}^{\flat}(f^{[Nd]})\not\equiv 0\pmod{3}\implies S_{\mathbb{N}_K}^{\flat}(f_d^{[Nd]})\ne 0\implies \log_{E_{d,3}}(P_K)\ne 0.$$ Now applying Proposition~\ref{p6.4}, we find that $E_{d,3}/K$ has rank exactly $1$.
\end{proof}

We assume the Galois representation $\rho_E$ modulo $3$ is reducible. By Proposition~\ref{p4.2}, we have that $\rho_E\cong \chi_M\oplus \chi_M\chi$ up to semisimplification.

\subsection{Sextic twists}
\label{subsec:sextic}
Consider the family of elliptic curves $y^2=x^3+c$ over $\mathbb{Q}$ for $c\in\mathbb{Q}$ up to isomorphism; denote this by $\mathcal{C}_c$. This family of elliptic curves has $j$-invariant $j=0$. The \textbf{sextic twist} by $D$, $g_{6,D}(\mathcal{C}_c)$, is the elliptic curve given by $$y^2=x^3+cD\cong_\mathbb{Q} (D^3y)^2=(D^2x)^3+cD^7\cong_\mathbb{Q} y^2=x^3+cD^7,$$ where $E_1\cong_F E_2$ denotes that $E_1$ is isomorphic to $E_2$ over $F$.  Thus the sextic twist by $D$ is a function $$g_{6,D}:\mathcal{C}_c\mapsto \mathcal{C}_{cD^7}.$$ The quadratic twist by $D$ on $\mathcal{C}_c$, denoted by $g_{2,D}(\mathcal{C}_c)$, is the curve $$Dy^2=x^3+c\cong_\mathbb{Q}(D^2y)^2=(Dx)^3+cD^3\cong_\mathbb{Q} y^2=x^3+cD^3,$$ so $$g_{2,D}(\mathcal{C}_c)=\mathcal{C}_{cD^3}.$$ The cubic twist by $D$ on $\mathcal{C}_c$, denoted by $g_{3,D}(\mathcal{C}_c)$, is the curve $$y^2=Dx^3+c\cong_\mathbb{Q}(Dy)^2=(Dx)^3+cD^2\cong_\mathbb{Q} y^2=x^3+cD^2,$$ so $$g_{3,D}(\mathcal{C}_c)=\mathcal{C}_{cD^2}.$$ We easily check that $g_{6,D}(\mathcal{C}_c)=\mathcal{C}_{cD^7}=g_{2,D}(\mathcal{C}_{cD^4})=g_{2,D}(g_{3,D^2}(\mathcal{C}_c))=g_{23,D}\circ g_{3,D}\circ g_{3,D}(\mathcal{C}_c)$.

As a result, we have that $$g_{6,D}=g_{2,D}\circ g_{3,D}\circ g_{3,D}=g_{3,D^2}\circ g_{2,D},$$ and it's clear that these functions commute.
\begin{prop}
\label{p8.1}
The family of curves $\mathcal{C}_c$ are exactly the elliptic curves which admit cubic twists.
\end{prop}
\begin{proof}
This family is precisely the family of elliptic curves with Weierstrass form $y^2=x^3+ax+b$ with $a=0$; in particular, this is exactly the family of elliptic curves with $j$-invariant $0$, since $j(E)=-1728\frac{(4a)^3}{\Delta}$, where $\Delta(E)=-16(4a^3+27b^2)$ (see \cite[p.~45]{silverman2009arithmetic}). 

On the other hand, an elliptic curve $E$ admits a cubic twist iff it has CM by $\mathbb{Q}(\sqrt{-3})=\mathbb{Q}(\zeta_3)$. Since the underlying field has characteristic $0$, by \cite[cor.~III.10.2]{silverman2009arithmetic}, this is equivalent to $\text{Aut}(E)=\mu_6=\langle \zeta_6\rangle \iff j(E)=0$.
\end{proof}

Due to this, the family $\mathcal{C}_c$ are the only elliptic curves which interest us.

\begin{prop}
\label{p8.2}
The Galois representation of any of the curves $\mathcal{C}_c$ is residually reducible modulo $3$.
\end{prop}
\begin{proof}
Since this family of curves admits cubic twists, they have CM by $L=\mathbb{Q}(\sqrt{-3})$. The prime $(3)$ is ramified in $L/\mathbb{Q}$, so let $(\sqrt{-3})\subset \mathcal{O}_L$ be the prime lying over $(3)$. Then $E[(\sqrt{-3})]\coloneqq \{x\in E| [a]x=0\forall  a\in (\sqrt{-3})\}=\{x\in E|[\sqrt{-3}]x=0\}$ is a group of order $3$. This group is defined over the Hilbert class field of $L$, which is $L$ since the class number of $L$ is $1$. Now, the group $\gal(L/\mathbb{Q})$ is generated by $\sigma$, the automorphism given by complex conjugation. Since $\sigma((\sqrt{-3}))=(\sqrt{-3})$, it follows that $E[(\sqrt{-3})]$ is defined over $L\cap \mathbb{R}=\mathbb{Q}$. As a result, $E[(\sqrt{-3})]\subsetneq E[3]$ is a subgroup preserved by $\gal(\overline{\mathbb{Q}}/\mathbb{Q})$, and thus $\rho_{\mathcal{C}_c}$ is residually reducible modulo $3$.
\end{proof}

Let $N=\text{cond}(\mathcal{C}_c)$ and let $D$ be some positive integer satisfying Assumption~\ref{assump:*}.
\begin{theorem}
\label{t8.3}
For a fixed $c$ (and thus $N$) and $D>0$ satisfying Assumption~\ref{assump:*} and $3\nmid h_{\mathbb{Q}(\sqrt{-3D})}$, the proportion of imaginary quadratic fields $K$ such that $g_{6,D}(\mathcal{C}_c)/K$ has rank $1$ is at least $2^{-1-\omega\left(ND/3^{v_3(ND)}\right)}$.
\end{theorem}
\begin{proof}
We have $\text{cond}(g_{6,D}(\mathcal{C}_c))=ND^t$ for some nonnegative integer $t$. Since $\text{cond}(g_{2,D}(\mathcal{C}_c))=ND^2$, it follows that the set of primes dividing $\text{cond}(g_{6,D}(\mathcal{C}_c))$ is a subset of the set of primes dividing $\text{cond}(g_{2,D}(\mathcal{C}_c))$. As a result, the subsequent cubic twist by $D^2$ yields an elliptic curve whose conductor does not have any new primes dividing it (compared to $ND^2$), and therefore does not require any more $\ell$-depletions.

As a result, any $K$ which admits a quadratic twist of $\mathcal{C}_c$ by $D$ will also admit a cubic twist by $D^2$. This occurs when $\Delta_K$ is a quadratic residue modulo all primes $\ell|ND$, and by Theorem~\ref{t6.6}, occurs for at least $(1/2)^{1+\omega\left(ND/3^{v_3(ND)}\right)}$ of the $\Delta_K$. Now applying Proposition~\ref{p6.4}, we conclude that every such $K$ also satisfies the property that $g_{6,D}(\mathcal{C}_c)/K$ has rank $1$.
\end{proof}
In particular, since every $\mathcal{C}_c$ is isomorphic to the sextic twist of $\mathcal{C}_1$ by $c$ (over a sufficient $K$), it is of particular interest to study $\mathcal{C}_1\coloneqq y^2=x^3+1$. Thus specializing Theorem~\ref{t8.3}, we have
\begin{corollary}
\label{c8.3.1}
For fixed $D>0$ with $3\nmid D$ and $3\nmid h_{\mathbb{Q}(\sqrt{-3D})}$, the proportion of imaginary quadratic fields $K$ such that $g_{6,D}(\mathcal{C}_1)/K$ has rank $1$ is at least $2^{-1-\omega(D)}$.
\end{corollary}
\begin{proof}
By applying Theorem~\ref{t8.3} with $N=\text{cond}(\mathcal{C}_1)=27$, the result follows.
\end{proof}
Addressing the $D<0$ case, we have the analogous results.
\begin{theorem}
\label{t8.4}
For a fixed $c$ (and thus $N$) and $D<0$ satisfying Assumption~\ref{assump:*} and $3\nmid h_{\mathbb{Q}(\sqrt{D})}$, the proportion of imaginary quadratic fields $K$ such that $g_{6,D}(\mathcal{C}_c)/K$ has rank $1$ is at least $2^{-1-\omega\left(ND/3^{v_3(ND)}\right)}$.
\end{theorem}
\begin{proof}
We have $\text{cond}(g_{6,D}(\mathcal{C}_c))=\pm ND^t$ for some nonnegative integer $t$. Since $\text{cond}(g_{2,D}(\mathcal{C}_c))=ND^2$, it follows that the set of primes dividing $\text{cond}(g_{6,D}(\mathcal{C}_c))$ is a subset of the set of primes dividing $\text{cond}(g_{2,D}(\mathcal{C}_c))$. As a result, the subsequent cubic twist by $D^2$ yields an elliptic curve whose conductor does not have any new primes dividing it (compared to $ND^2$), and therefore does not require any more $\ell$-depletions.

As a result, any $K$ which admits a quadratic twist of $\mathcal{C}_c$ by $D$ will also admit a cubic twist by $D^2$. This occurs when $\Delta_K$ is a quadratic residue modulo all primes $\ell|ND$, and by Theorem~\ref{t6.7}, occurs for at least $2^{-1-\omega\left(ND/3^{v_3(ND)}\right)}$ of the $\Delta_K$. Now applying Proposition~\ref{p6.4}, we conclude that every such $K$ also satisfies the property that $g_{6,D}(\mathcal{C}_c)/K$ has rank $1$.
\end{proof}
Once again specializing to $\mathcal{C}_1$, we have:
\begin{corollary}
For fixed $D<0$ with $3\nmid D$ and $3\nmid h_{\mathbb{Q}(\sqrt{D})}$, the proportion of imaginary quadratic fields $K$ such that $g_{6,D}(\mathcal{C}_1)/K$ has rank $1$ is at least $2^{-1-\omega(D)}$.
\end{corollary}
\begin{proof}
By applying Theorem~\ref{t8.4} with $N=\text{cond}(\mathcal{C}_1)=27$, the result follows.
\end{proof}

\section{Ranks of twists over $\mathbb{Q}$}
\label{sec:ranksoverQ}
For some suitable elliptic curve $E/\mathbb{Q}$, we have discussed the proportion of imaginary quadratic fields $K$ with $E_D/K$ yielding elliptic curves of either rank $1$ or rank $0$. We will now consider the ranks over $\mathbb{Q}$ instead.

We will need the concept of a root number. The \textbf{root number} $w_{E/K}$ of an elliptic curve $E/K$ is the value $w_{E/K}\in \{-1,1\}$ such that $L_{E/K}(s)=w_{E/K}L_{E/K}( 2-s)$.

\begin{theorem}
\label{t9.1}
Fix $E$ with $N=\text{cond}(E)$ and $D$ satisfying $\gcd(N,D)=\gcd(N,6)=1$. Then $E_D/\mathbb{Q}$ has rank $1$ for at least $\frac{\phi(N)}{4N}$ of all such $D$, and rank $0$ for at least $\frac{\phi(N)}{4N}$ of all such $D$.
\end{theorem}
\begin{proof}
By \cite{kolyvagin1989finiteness}, Heegner points in $E/K$ exist iff $w_{E/K}=-1$. Furthermore, if $E/K$ has rank $1$, then $E/\mathbb{Q}$ has rank $\frac{1-w_{E/\mathbb{Q}}}{2}$, and $w_{E_D/\mathbb{Q}}=\left(\frac{D}{-N}\right)w_{E/\mathbb{Q}}$. It follows that if $E_D/K$ has rank $1$, then $E_D/\mathbb{Q}$ has rank $1$ if $w_{E/\mathbb{Q}}=-1$.

For $D>0$, Corollary~\ref{c6.6.1} shows that when $3\nmid h_{\mathbb{Q}(\sqrt{-3D})}$, there exists some imaginary quadratic field $K$ (in fact, a positive density) such that $E_D/K$ has rank $1$, and thus it suffices to check when $w_{E/\mathbb{Q}}=-1$. Since $w_{E/\mathbb{Q}}$ is fixed, we check the proportion of $D>0$ such that $\left(\frac{D}{-N}\right)=\pm 1$ in each case. We have $\left(\frac{D}{-N}\right)=\left(\frac{D}{N}\right)$ which depends only on the residue of $D$ modulo $N$. There are exactly $\phi(N)/2$ quadratic residues and quadratic nonresidues, and thus the proportion of $D$ (assuming $3\nmid h_{\mathbb{Q}(\sqrt{-3D})}$) is exactly $\frac{\phi(N)}{2N}$. Now \cite[Lemma~2.2]{byeon2004class} implies that for every $m$ such that $D\equiv m\pmod{N}\implies \left(\frac{D}{-N}\right)=-w_{E/\mathbb{Q}}$, then the proportion $$S_-(X,-3m,N)\coloneqq \frac{\left|D>0\hspace{1mm}|\hspace{1mm} -3D\equiv -3m\pmod{N}, \hspace{1mm} h_{\mathbb{Q}(\sqrt{-3D})}\not\equiv 0\pmod{3}\right|}{\left|D>0\hspace{1mm}|\hspace{1mm} -3D\equiv -3m\pmod{N}\right|}\ge \frac{1}{2}.$$ Since this holds for every such $m$, it follows that the proportion of such $D$ satisfying $\left(\frac{D}{-N}\right)=-w_{E/\mathbb{Q}}$ is at least $\frac{1}{2}$.

For $D<0$, Corollary~\ref{c6.6.2} shows that when $3\nmid h_{\mathbb{Q}(\sqrt{D})}$, there exists some imaginary quadratic field $K$ (in fact, a positive density) such that $E_D/K$ has rank $1$, and thus it suffices to check when $w_{E/\mathbb{Q}}=-1$. Since $w_{E/\mathbb{Q}}$ is fixed, we check the proportion of $D<0$ such that $\left(\frac{D}{-N}\right)=\pm 1$ in each case. We have $\left(\frac{D}{-N}\right)=\left(\frac{D}{N}\right)\left(\frac{D}{-1}\right)=-\left(\frac{D}{N}\right)$ which depends only on the residue of $D$ modulo $N$.  There are exactly $\phi(N)/2$ quadratic residues and quadratic nonresidues, and thus the proportion of $D$ (assuming $3\nmid h_{\mathbb{Q}(\sqrt{D})}$) is exactly $\frac{\phi(N)}{2N}$. Now \cite[Lemma~2.2]{byeon2004class} implies that for every $m$ such that $D\equiv m\pmod{N}\implies \left(\frac{D}{-N}\right)=-w_{E/\mathbb{Q}}$, then the proportion $$S_-(X,m,N)\coloneqq \frac{\left|D<0\hspace{1mm}|\hspace{1mm} D\equiv m\pmod{N}, \hspace{1mm} h_{\mathbb{Q}(\sqrt{D})}\not\equiv 0\pmod{3}\right|}{\left|D<0\hspace{1mm}|\hspace{1mm} D\equiv m\pmod{N}\right|}\ge \frac{1}{2}.$$ Since this holds for every such $m$, it follows that the proportion of such $D$ satisfying $\left(\frac{D}{-N}\right)=-w_{E/\mathbb{Q}}$ is at least $\frac{1}{2}$.

We conclude that in either case, the proportion of $D$ with $E_D/\mathbb{Q}$ having rank $1$ is at least $\frac{\phi(N)}{2N}\cdot \frac{1}{2}=\frac{\phi(N)}{4N}$. Analogously, when $\left(\frac{D}{-N}\right)=w_{E/\mathbb{Q}}$, we find that $E_D/\mathbb{Q}$ has rank $0$, and the same result holds.
\end{proof}

Noting that the assumptions hold for a fixed (positive) proportion of $D$, we conclude that for elliptic curves satisfying the above assumptions, Conjecture~\ref{weakgoldfeld} holds.
\begin{corollary}
For such $E$ satisfying the assumptions of Theorem~\ref{t9.1}, the weak Goldfeld conjecture holds.
\end{corollary}

\bibliographystyle{alpha}
\bibliography{arxiv}

\end{document}